\documentclass[11pt]{amsart}

\usepackage{amsmath}
\usepackage{amssymb}
\usepackage{amsthm}
\usepackage{enumerate} 
\usepackage{amsthm}
\usepackage{graphicx}
\usepackage{psfrag}
\usepackage{framed}
\usepackage[T1]{fontenc}
\input{xy}  
\xyoption{all} 
\CompileMatrices 

\theoremstyle{plain}
\newtheorem{theorem}{Theorem}[section]
\newtheorem{lemma}[theorem]{Lemma}
\newtheorem{corollary}[theorem]{Corollary}
\newtheorem{proposition}[theorem]{Proposition}

\theoremstyle{definition}

\newtheorem{question}[theorem]{Question}
\newtheorem{definition}[theorem]{Definition}
\newtheorem{remark}[theorem]{Remark}

\numberwithin{equation}{section}

\newcommand{\const}{{\tilde k}}

\newcommand{\ad}{\text{ad}}
\newcommand{\ct}{\mathbb T}
\newcommand{\tr}{\mathrm{tr}}

\allowdisplaybreaks

\begin{document}
\title{Perturbations of C$^*$-algebraic invariants}

\author[E. Christensen]{Erik Christensen}
\address{\hskip-\parindent
Erik Christensen, Department of Mathematical Sciences, University of Copenhagen, Copenhagen, Denmark.}
\email{echris@math.ku.dk}

\author[A. Sinclair]{Allan Sinclair}
\address{\hskip-\parindent
Allan Sinclair, School of Mathematics, University of Edinburgh, Edinburgh, EH9 3JZ, UK.}
\email{a.sinclair@ed.ac.uk}

\author[R. Smith]{Roger R. Smith}
\address{\hskip-\parindent
Roger Smith, Department of Mathematics, Texas A{\&}M University,
College Station TX 77843-3368, U.S.A.}
\thanks{The authors gratefully acknowledge funding
 from the Edinburgh Mathematical Society and the National
 Science Foundation for partial support of this research}
\email{rsmith@math.tamu.edu}

\author[S. White]{Stuart White}
\address{\hskip-\parindent
Stuart White, Department of Mathematics, University of Glasgow, Glasgow Q12 8QW, UK.}
\email{s.white@maths.gla.ac.uk}

\date{\today}

\begin{abstract}
Kadison and Kastler introduced a metric on the set of all
C$^*$-algebras on a fixed Hilbert space. In this paper structural
properties of C$^*$-algebras which are close in this metric are examined. Our main result is that the property of having a positive answer
to Kadison's similarity problem transfers to close C$^*$-algebras.
In
establishing this result we answer questions about closeness of
commutants and tensor products when one algebra satisfies the
similarity property. We also examine $K$-theory and traces of close
C$^*$-algebras, showing that sufficiently close algebras have
isomorphic Elliott invariants when one algebra has the similarity
property.
\end{abstract}

\maketitle

\section{Introduction}

In \cite{Kadison.Kastler}, Kadison and Kastler introduced the study of uniform perturbations of operator algebras.  They considered a fixed C$^*$-algebra $C$ and equipped the set of all C$^*$-subalgebras of $C$ with a metric arising from  Hausdorff distance between the unit balls of these subalgebras.  In general terms, two C$^*$-subalgebras $A$ and $B$ of $C$ are close if elements from the unit ball of $A$ can be approximated well in the unit ball of $B$, and vice versa.  A precise definition will be given in Section \ref{Sim} below. Kadison and Kastler conjectured that sufficiently close subalgebras must be isomorphic and that this isomorphism should be spatially implemented when $C$ is faithfully represented on some Hilbert space.  In the 1970's and 1980's various cases of this conjecture were established: \cite{Raeburn.CohomologyPerturbation} resolves the problem when one algebra is an injective von Neumann algebra (see also \cite{Christensen.Perturbations1,Christensen.NearInclusions}); \cite{Christensen.NearInclusions} solves the problem when one algebra is separable and AF (see also \cite{Phillips.PerturbationAF}); \cite{Phillips.Perturbations2} examines the situation for continuous trace algebras and \cite{Khoshkam.UnitaryEquivalence} looks at extensions of some of the cases from \cite{Phillips.PerturbationAF,Christensen.NearInclusions,Phillips.Perturbations2}; and \cite{Johnson.NearInclusions} examines sub-homogeneous C$^*$-algebras. Recent progress has been made in \cite{Saw.PerturbNuclear} which gives a positive answer to the question when one algebra is separable and nuclear. In full generality \cite{Christensen.CounterExamples} provides 
counterexamples to the conjecture. These counterexamples are non-separable C$^*$-algebras and the problem remains open when $A,B$ are von Neumann algebras or separable C$^*$-algebras.
In the absence of a general isomorphism result, a naturally
arising question is whether close C$^*$-algebras must share
the same invariants. This will be a continuing theme of
the paper.  In this introduction we discuss our results in qualitative terms. Precise estimates will be given in the main text.

The principal objective of this article is to examine connections between the theory of perturbations and Kadison's similarity problem. Kadison's similarity problem was set out in \cite{Kadison.OrthRepresentations} and asks whether every bounded unital representation from a unital C$^*$-algebra $A$ into $\mathbb B(\mathcal H)$ is similar to a $^*$-representation of $A$ on $\mathcal H$.  In \cite{Haagerup.SimilarityCyclic}, Haagerup gave a positive answer to this question for cyclic representations and showed that a bounded representation $\pi$ of a C$^*$-algebra on $\mathbb B(\mathcal H)$ is similar to a $^*$-representation if and only if $\pi$ is completely bounded.  We say that $A$ has the {\emph{similarity property}} if the similarity problem has a positive answer for $A$. In \cite{Kirchberg.DerivationSimilarity}, Kirchberg showed that $A$ has the similarity property if and only if the \emph{derivation problem} also has a positive answer for $A$, that is given a $^*$-representation $\pi:A\rightarrow\mathbb B(\mathcal H)$ and a bounded $\pi$-derivation $\delta:A\rightarrow \mathbb B(\mathcal H)$, there is some $x\in\mathbb B(\mathcal H)$ such that $\delta(a)=[x,\pi(a)]=x\pi(a)-\pi(a)x$ for all $a\in A$. Such derivations are called \emph{inner}. 
There is another equivalent formulation that we now discuss.

Motivated by the similarity problem, Pisier  introduced
the notion of the {\emph{length}} $\ell(A)$ of an operator
algebra $A$ in \cite{Pisier.StPetersburg} and examined its properties in \cite{Pisier.SimilarityRemarks,Pisier.SimilarityNuclear}. This integer arises from the ability to write
matrices over $A$ as products of bounded length, where the
constituent factors  alternate between scalar matrices and
diagonal matrices over $A$ (the precise details are given in
Definition \ref{Def.Length}). If such decompositions do not exist then
$\ell(A)=\infty$, although no examples of this are
currently known. An easy consequence of finite length is
that all bounded homomorphisms of $A$ into any $\mathbb B(\mathcal
H)$ are completely bounded, which solves the similarity
problem for such algebras, and is indeed equivalent to it.
Remarkably, nuclearity is characterised by $\ell(A)\leq 2$ \cite{Pisier.SimilarityNuclear},
while all C$^*$-algebras lacking tracial states have 
length at most 3. These results are surveyed in Pisier's  monograph
\cite{Pisier.SimilarityBook}. For our purposes, the finite length property will be
a
convenient formulation of the similarity problem, and we
will be able to show that this property transfers to nearby
C$^*$-algebras. This also uses a more technical characterisation called the distance property, described below in Definition \ref{Prelim.LDK}. 

There are two open questions concerning the behaviour of the distance between algebras under standard constructions which  arise from \cite{Christensen.NearInclusions} and  are connected to the similarity property.  Given two C$^*$-algebras $A$ and $B$ on some Hilbert space $\mathcal H$, with $A$ and $B$ close, must the commutants $A'$ and $B'$ be close?  Under the same hypothesis, must the algebras $A\otimes E$ and $B\otimes E$ be close (as subalgebras of $\mathbb B(\mathcal H)\otimes E$) for any nuclear C$^*$-algebra $E$?  The work of \cite{Christensen.NearInclusions} gives positive answers to these questions provided, in today's language, both $A$ and $B$ satisfy the similarity property.  In section \ref{Length} we show that if $A$ has the similarity property and $B$ is sufficiently close to $A$, then $B$ also has the similarity property (with constants depending on the similarity length and length constant). To do this, we initially answer the first question above regarding closeness of commutants when only one algebra has the similarity property.  As a consequence, we also obtain a positive answer when one algebra has the similarity property. 

Khoshkam examined the $K$-theory of close C$^*$-algebras in \cite{Khoshkam.PerturbationK}, showing that there is a natural isomorphism between the ordered $K$-theories of sufficiently close nuclear C$^*$-algebras.  The key ingredient required for \cite{Khoshkam.PerturbationK} was that if $A$ and $B$ are close and nuclear, then the matrix algebras $\mathbb M_n(A)$ and $\mathbb M_n(B)$ are uniformly close (so that the distance between these algebras is bounded independently of $n$).  Khoshkam's isomorphism can be defined whenever this condition holds.  In particular, we  show in Corollary \ref{K.Sim} that sufficiently close C$^*$-algebras have isomorphic ordered $K$-theories provided that one algebra has the similarity property. The distance we require depends on the similarity length and constant of this algebra.  

Khoshkam's work opens the possibility of using results from Elliott's classification programme to address perturbation questions.  We discuss this topic in sections \ref{KTrace} and \ref{Kirchberg}, with the objective of showing that invariants and properties used in the classification programme transfer to sufficiently close algebras. In Lemma \ref{KTrace.Lemma} we construct an affine isomorphism between the traces on sufficiently close C$^*$-algebras. When one algebra has the similarity property, this isomorphism and the isomorphism between $K$-theories from Corollary \ref{K.Sim} respect the natural pairing between the $K_0$ and the traces. In particular there is an isomorphism between the Elliott invariants of sufficiently close nuclear C$^*$-algebras.

Section \ref{Kirchberg} gives an example of how the classification programme can be used to quickly give perturbation results.  We use Kirchberg and C. Phillips' classification of Kirchberg algebras (simple, separable, purely infinite and nuclear C$^*$-algebras) \cite{Kirchberg.Phillips,Kirchberg.ClassificationBook} to show that any C$^*$-algebra satisfying the UCT which is sufficiently close to a Kirchberg algebra with the UCT is necessarily isomorphic to it.  Given earlier results, it suffices to examine how the property of being purely infinite behaves under perturbations and we show that a C$^*$-algebra that is close to a simple and purely infinite one is also purely infinite. We do this by showing that the property of being real rank zero also transfers to sufficiently close algebras.  As in the previous section, we establish these results in as much generality as possible, not just in the nuclear setting.

The paper is structured as follows.
In section \ref{Sim} we recall the precise definition of the metric introduced by Kadison and Kastler in \cite{Kadison.Kastler} and give a detailed account of  how the similarity property gives rise to results in the theory of perturbations.  In section \ref{Tech} we establish some technical preliminaries required in our later work. In particular, we examine the behaviour of the centre valued trace and coupling constants in the context of close von Neumann algebras. These play important technical roles in section \ref{Length}, where we establish our main result that algebras close to those of finite length again have finite length and discuss its consequences. Section \ref{KTrace} examines the $K$-theory and traces of close C$^*$-algebras, while Section \ref{Kirchberg} contains our example of how the classification programme gives rise to perturbation results.  The paper ends in Section \ref{Questions} with a brief collection of open problems.

\subsection*{Acknowledgments} The authors would like to thank Joachim Zacharias for bringing \cite{Khoshkam.PerturbationK} to their attention and the referees for their careful reading of this paper and useful comments.

\section{Similarity Length and Perturbations}\label{Sim}
This section fills in the quantitive versions of the definitions from the introduction and examines the connections between perturbation theory and the similarity problem from the literature. We begin by recalling the definition of the metric $d$ on the collection of all C$^*$-subalgebras of a fixed C$^*$-algebra from \cite{Kadison.Kastler} and the notion of a near inclusion from \cite{Christensen.NearInclusions}.
\begin{definition}\label{Prelim.DefMetric}
Let $A$ and $B$ be C$^*$-subalgebras of some C$^*$-algebra $C$. Define $d(A,B)$ to be the infimum of all $\gamma>0$ with the property that given $x$ in the unit ball of $A$ or $B$, there exists $y$ in the unit ball of the other algebra with $\|x-y\|<\gamma$.
\end{definition}

\begin{definition}\label{Prelim.DefNear}
Let $A$ and $B$ be C$^*$-subalgebras of some C$^*$-algebra $C$ and let $\gamma>0$.  Write $A\subseteq_\gamma B$ if for each $x$ in the unit ball of $A$ there is $y\in B$ with $\|x-y\|\leq \gamma$.  Write $A\subset_\gamma B$ if $A\subseteq_{\gamma'}B$ for some $\gamma'<\gamma$.
\end{definition}

Note that  Definition \ref{Prelim.DefNear} does not require that $y$ lie in the unit ball of $B$. This means that the notion of distance between two C$^*$-subalgebras $A$ and $B$ defined by considering the infimum of all $\gamma$ for which $A\subseteq_\gamma B$ and $B\subseteq_\gamma A$ does not obviously satisfy the triangle inequality.  The proposition below sets out the relationships between the concepts of Definitions \ref{Prelim.DefMetric} and \ref{Prelim.DefNear}.
All are immediate consequences of the definitions and so we omit their proofs.

\begin{proposition}\label{Prelim.MetricNear}
Let $A,B$ and $C$ be $C^*$-subalgebras of some C$^*$-algebra $E$. 
\begin{itemize}
\item[\rm (i)] If $A\subseteq_\gamma B$ and $B\subseteq_\delta C$, then $A\subseteq_{\gamma+\delta(1+\gamma)}C$.\label{Prelim.MetricNear.1}
\item[\rm (ii)] If $d(A,B)\leq\gamma$, then $A\subseteq_\gamma B$ and $B\subseteq_\gamma A$.\label{Prelim.MetricNear.2}
\item[\rm (iii)] If $A\subseteq_\gamma B$ and $B\subseteq_\gamma A$, then $d(A,B)\leq 2\gamma$.\label{Prelim.MetricNear.3}
\end{itemize}
\end{proposition}

In general it is unknown whether a near inclusion $A\subset_\gamma B$ of two C$^*$-algebras on some Hilbert space $\mathcal H$ induces a near inclusion $B'\subset_{L\gamma}A'$ between the commutants
for a suitably chosen constant $L$.  Based on \cite{Christensen.Perturbations2}, a distance property $D_k$ was introduced in \cite[Definition 2.2]{Christensen.NearInclusions} which allows such a deduction to be made.  Subsequently it was shown in \cite{Christensen.ExtensionDerivations,Christensen.ExtensionDerivations2} that a C$^*$-algebra has such a distance property if and only if for every representation $\pi \colon A\to 
\mathbb B(\mathcal H)$, every derivation from $\pi(A)$ into $\mathbb B(\mathcal H)$ is inner. We now review this connection.

Let $A\subset\mathbb B(\mathcal H)$ be a C$^*$-algebra.  Given $x\in\mathbb B(\mathcal H)$ we can define a derivation $\ad(x)|_A$ on $A$ by $\ad(x)|_A(a)=[x,a]=xa-ax$. The Arveson distance formula \cite{Arveson.InterpolationNest} gives
\begin{equation}\label{Arveson}
d(x,A')=\frac{1}{2}\|\ad(x)|_A\|_{\text{cb}},\quad x\in\mathbb B(\mathcal H),
\end{equation}
see also \cite[Proposition 2.1]{Christensen.Perturbations2}. Theorem 3.2 of \cite{Christensen.ExtensionDerivations2} shows that every derivation of $A$ into $\mathbb B(\mathcal H)$ is inner (i.e. of the form $\ad(x)|_A$ for some $x\in\mathbb B(\mathcal H)$) if and only if there is some $k>0$ such that
\begin{equation}\label{PropDK}
d(x,A')\leq k\|\ad(x)|_A\|,\quad x\in\mathbb B(\mathcal H).
\end{equation}
Using the distance formula in (\ref{Arveson}), it follows that every derivation of $A$ into $\mathbb B(\mathcal H)$ is inner if and only if there is some $k>0$ such that
\begin{equation}
\|\ad(x)|_A\|_{\text{cb}}\leq 2k\|\ad(x)|_A\|,\quad x\in\mathbb B(\mathcal H).
\end{equation}
We formalise these concepts in the following definitions, the latter being \cite[Definition 2.2]{Christensen.NearInclusions}.

\begin{definition}\label{Prelim.LDK}
Let $k>0$ and let $A$ be a C$^*$-algebra.  A representation $\pi$ of $A$ on $\mathcal H$ has the \emph{local distance property} $LD_k$ if
\begin{equation}\label{Prelim.LDK.1}
d(x,\pi(A)')\leq k\|\ad(x)|_{\pi(A)}\|,\quad x\in\mathbb B(\mathcal H).
\end{equation}
If every representation of $A$ has the local distance property $LD_k$, then $A$ has the \emph{distance property} $D_k$.$\hfill\square$
\end{definition}
By the preceding discussion, $A$ has the distance property $D_k$ for some $k>0$ if and only if the derivation problem has a positive answer for $A$. Furthermore a representation $\pi$ of $A$ on $\mathcal H$ has the local distance property $LD_k$ for some $k>0$ if and only if every $\pi$-derivation is inner. A near inclusion $A\subseteq_\gamma B$ of C$^*$-algebras on a Hilbert space $\mathcal H$ induces a near inclusion of $B'$ into $A'$ when $A$ has the local distance property on $\mathcal H$. This is easily established in the proposition below. The proof is extracted from the proof of \cite[Theorem 3.1]{Christensen.NearInclusions}.
\begin{proposition}\label{LDK.Comm}
Let $A,B\subset\mathbb B(\mathcal H)$ be C$^*$-algebras with $A\subseteq_\gamma B$.  If $A\subseteq\mathbb B(\mathcal H)$ has the local distance property $LD_k$, then 
\begin{equation}\label{LDK.Comm.1}
B'\subseteq_{2k\gamma}A'.
\end{equation}
\end{proposition}
\begin{proof}
Fix $x\in B'$. For $a$ in the unit ball of $A$, there is some $b\in B$ with $\|a-b\|\leq\gamma$. Then $\|[x,a]\|=\|[x,(a-b)]\|\leq 2\|x\|\gamma$. Property $LD_k$ gives $d(x,A')\leq 2k\|x\|\gamma$ and hence the near inclusion (\ref{LDK.Comm.1}).
\end{proof}

Corollary 5.4 of \cite{Christensen.ExtensionDerivations2} shows that cyclic representations of C$^*$-algebras have the local distance property $LD_{12}$ and hence solves the derivation problem for cyclic representations.  We need to consider representations with a finite set of cyclic vectors in Section \ref{Length}. The next proposition is an easy extension of \cite[Corollary 5.4]{Christensen.ExtensionDerivations2} to this case.
\begin{proposition}\label{LDK.Cyclic}
Let $\pi$ be a representation of a C$^*$-algebra $A$ on a Hilbert space $\mathcal H$. If $\pi(A)$ has a finite cyclic set of $m$ vectors, then $\pi$ has the local distance property $LD_{12m}$.
\end{proposition}
\begin{proof}
Let $\xi_1,\ldots, \xi_m$ in $\mathcal H$ be a cyclic set
for $\pi(A)$. Then $(\xi_1,\ldots,\xi_m)^T\in \mathcal
H\otimes \mathbb C^m$ is a cyclic vector for $\pi(A)\otimes
\mathbb M_m$. Fix $y\in \mathbb B(\mathcal H)$ and let
${\text{ad}}(y)|_{\pi(A)}$ be the associated derivation on
$\pi(A)$. Then ${\text{ad}}(y\otimes I_{{\mathbb M_m}})|_{\pi(A)\otimes
\mathbb M_m}$  satisfies
\begin{equation}\label{eq12}
d(y\otimes I_{{\mathbb M_m}},\pi(A)'\otimes I_{{\mathbb M_m}})\leq
12\,\|{\text{ad}}(y\otimes
I_{{\mathbb M_m}})|_{\pi(A)\otimes \mathbb M_m}\|
\end{equation}
using \cite[Cor. 5.4]{Christensen.ExtensionDerivations2}, which is valid for algebras with cyclic vectors.
Since ${\text{ad}}(y\otimes I_{{\mathbb M_m}})={\text{ad}}(y)\otimes id_{{\mathbb M_m}}$,
the estimate
\begin{equation}
d(y,\pi(A)')\leq 12\,m\,\|{\text{ad}}(y)|_{\pi(A)}\|
\end{equation}
follows from \eqref{eq12} and the general inequality $\|\phi\otimes id_{{\mathbb M_m}}\|\leq m\|\phi\|$ for bounded maps $\phi$ between C$^*$-algebras, which is \cite[Exercise 3.10]{Paulsen.CB.Book}. This shows that we have property $LD_{12m}$.
\end{proof}

We now turn to Pisier's notion  of the length of an operator algebra, \cite{Pisier.StPetersburg}.

\begin{definition}\label{Def.Length}
Let $A$ be a C$^*$-algebra faithfully represented on $\mathbb B(\mathcal H)$.  Say that $A$ has length at most $\ell$ if there exists a constant $K>0$ such that for each $n\in\mathbb N$ and $x\in \mathbb M_n(A)$, there is an integer $N$, diagonal matrices $d_1,\dots,d_\ell\in \mathbb M_N(A)$ and scalar matrices $\lambda_0\in \mathbb M_{n, N},\lambda_1,\dots,\lambda_{\ell-1}\in \mathbb M_N,\lambda_\ell\in \mathbb M_{N,n}$ such that
\begin{equation}
x=\lambda_0d_1\lambda_1d_1\lambda_2\dots \lambda_{\ell-1}d_\ell\lambda_\ell
\end{equation}
and
\begin{equation}
\prod_{i=0}^\ell\|\lambda_i\|\prod_{i=1}^\ell\|d_i\|\leq K\|x\|.
\end{equation}
In this case we say that $A$ has \emph{length constant at most $K$}.$\hfill\square$
\end{definition}
It is easy to see that this definition does not depend on the choice of the faithful representation of $A$, but phrasing it in this fashion ensures that we do not have to distinguish between the unital and non-unital cases.  In \cite{Haagerup.SimilarityCyclic}, it is shown that a unital C$^*$-algebra $A$ has the similarity property if and only if there exists some $d\geq 1$ and positive constant $K'$ such that 
\begin{equation}\label{Sim.Degree}
\|u\|_{\text{cb}}\leq K'\|u\|^d,
\end{equation}
for all bounded unital homomorphisms $u:A\rightarrow\mathbb B(\mathcal H)$. In \cite{Pisier.StPetersburg}, Pisier shows that this happens if and only if $A$ has finite length. Furthermore, the infimum over all $d$ for which there is a constant $K'$ so that (\ref{Sim.Degree}) holds is precisely the length of $A$. One direction is easy to see: if $A$ has length at most $\ell$ and length constant $K$, then (\ref{Sim.Degree}) holds with $K'=K$ and $d=\ell$.   Note too that while (\ref{Sim.Degree}) implies that $A$ has length at most $\lfloor d\rfloor$, it does not give us information about the length constant of $A$. For more information on this topic we refer the reader to Pisier's monograph on similarity problems \cite{Pisier.SimilarityBook} and  his operator space text \cite[Chapter 27]{Pisier.OperatorSpaceBook}.

The next two propositions give quantified versions of the equivalence between the properties of satisfying the derivation problem and having finite length.  The first can be found in \cite[Section 4 (in particular Remark 4.7)]{Pisier.StPetersburg}, while the second is the derivation version of the calculation \cite[Proposition 10.6]{Pisier.SimilarityBook}. This is well known but we include the proof for completeness.
\begin{proposition}\label{Prelim.DKLength}
Let $A$ have property $D_k$ for some $k$.  Then the length of $A$ is at most $\lfloor 2k\rfloor$.
\end{proposition}

\begin{proposition}\label{Prelim.LengthDK}
Let $A$ be a C$^*$-algebra with length at most $\ell$ and length constant at most $K$. Then $A$ has property $D_k$ for $k=K\ell/2$.
\end{proposition}
\begin{proof}
Suppose that we are given a representation $\pi:A\rightarrow\mathbb B(\mathcal H)$. Fix $y\in\mathbb B(\mathcal H)$. Given $n\in\mathbb N$ and an operator $x\in\mathbb M_n(\pi(A))$,  let $N$, $\lambda_0,\dots,\lambda_{\ell}$ and $d_1,\dots,d_\ell$ be as in Definition \ref{Def.Length}. Then $\ad(y)\otimes\text{id}_{\mathbb M_n}=\ad(y\otimes I_{\mathbb M_n})$. Using the facts that $(y\otimes I_{\mathbb M_n})\lambda_0=\lambda_0(y\otimes I_{\mathbb M_N})$, $(y\otimes I_{\mathbb M_N})\lambda_{\ell}=\lambda_\ell(y\otimes I_{\mathbb M_n})$ and that $y\otimes I_{\mathbb M_N}$ commutes with each $\lambda_1,\dots,\lambda_{\ell-1}$, we can apply Leibnitz's rule to obtain
\begin{equation}
\ad(y\otimes I_{\mathbb M_n})(x)=\sum_{i=1}^{l}\lambda_0d_1\lambda_1\dots\lambda_i[(y\otimes I_{\mathbb M_N}),d_i]\lambda_{i+1}d_{i+1}\dots\lambda_{\ell-1}d_\ell\lambda_\ell.
\end{equation}
Therefore
\begin{align}
\|\ad(y)&\otimes\text{id}_{\mathbb M_n}(x)\|\notag\\
&\leq\sum_{i=1}^\ell\|\lambda_0\|\|d_1\|\|\lambda_1\|\dots\|\lambda_i\|\|[(y\otimes I_{\mathbb M_N}),d_i]\|\|\lambda_{i+1}\|\|d_{i+1}\|\dots\|\lambda_{\ell-1}\|\|d_\ell\|\|\lambda_\ell\|\notag\\
&\leq\ell \|\ad(y)|_{\pi(A)}\|\prod_{i=0}^\ell\|\lambda_i\|\prod_{i=1}^\ell\|d_i\|\leq K\ell\|\ad(y)|_{\pi(A)}\|\|x\|.
\end{align}
The result follows from (\ref{Arveson}).
 \end{proof}

We can use the factorisations of Definition \ref{Def.Length} to lift near inclusions $A\subseteq_\gamma B$ to near inclusions $A\otimes \mathbb M_n\subseteq_{L\gamma} B\otimes\mathbb M_n$ when $A$ has finite length.  The next proposition has been known to Pisier for some time and is the similarity length version of \cite[Theorem 3.1]{Christensen.NearInclusions} which obtains an analogous result for algebras using property $D_k$.
\begin{proposition}\label{Length.Matrix}
Let $A,B\subset\mathbb B(\mathcal H)$ be C$^*$-algebras with $A\subseteq_\gamma B$ for some $\gamma>0$.  Suppose that $A$ has length at most $\ell$ and length constant at most $K$.  Then $A\otimes\mathbb M_n\subseteq_\mu B\otimes\mathbb M_n$  for all $n\in\mathbb N$, where $\mu$ is given by
\begin{equation}\label{Prelim.DefMu}
\mu=K((1+\gamma)^\ell-1).
\end{equation}
\end{proposition}
\begin{proof}
Fix $n\in\mathbb N$ and identify $A\otimes\mathbb M_n$ and $B\otimes\mathbb M_n$ with $\mathbb M_n(A)$ and $\mathbb M_n(B)$ respectively.  Take $x$ in the unit ball of $\mathbb M_n(A)$ and find scalar matrices $\lambda_0,\dots,\lambda_\ell$ and diagonal matrices $d_1,\dots,d_\ell$ as in Definition \ref{Def.Length}.  For each diagonal matrix $d_i\in \mathbb M_N(A)$, we can apply the near inclusion $A\subseteq_{\gamma}B$ to each entry to produce a diagonal matrix $e_i\in \mathbb M_N(B)$ with $\|d_i-e_i\|\leq\gamma\|d_i\|$. Then
\begin{equation}
y=\lambda_0e_1\lambda_1e_2\lambda_2\dots\lambda_{\ell-1}e_\ell\lambda_\ell
\end{equation}
defines an element of $\mathbb M_n(B)$ and an inductive calculation gives
\begin{equation}
\|x-y\|\leq K\Big(\gamma+\gamma(1+\gamma)+\gamma(1+\gamma)^2+\dots+\gamma(1+\gamma)^{\ell-1}\Big)=K\Big((1+\gamma)^\ell-1\Big)=\mu,
\end{equation}
which completes the proof.
\end{proof}

\begin{remark}
There is also a version of Proposition \ref{Length.Matrix} for finite sets which we state here for use in \cite{Saw.PerturbNuclear}.  Suppose that $A,B\subset\mathbb B(\mathcal H)$ are C$^*$-algebras and that $A$ has length at most $\ell$ and length constant $K$.  Given any $n\in\mathbb N$ and finite set $X$ in the unit ball of $A\otimes\mathbb M_n$, there exists a finite set $Y$ in the unit ball of $A$ such that if $Y\subseteq_\gamma B$ for some $\gamma>0$ (by which we mean that for each $y\in Y$, there is some $b\in B$ with $\|y-b\|\leq\gamma$), then $X\subseteq_\mu B\otimes\mathbb M_n$, where $\mu=K((1+\gamma)^\ell-1)$.  Note that the set $Y$ consists of all the entries of the diagonal matrices $d_i$ in the proof of Proposition \ref{Length.Matrix} and so depends only on $X$ (and not on $B$ or the value of $\gamma$).
\end{remark}

The next corollary follows from Proposition \ref{Length.Matrix} using the completely positive approximation property for nuclear C$^*$-algebras \cite{Choi.EffrosCPAP}.  The proof is identical to the deduction of Theorem 3.1 of \cite{Christensen.NearInclusions} from equation (3) on page 253 of \cite{Christensen.NearInclusions} and so is omitted.
\begin{corollary}\label{Length.TensorNuclear}
Let $A,B\subset\mathbb B(\mathcal H)$ be C$^*$-algebras with $A\subset_\gamma B$ for some $\gamma>0$. Suppose that $A$ has length at most $\ell$ and length constant at most $K$.  Given any nuclear C$^*$-algebra $E$, we have $A\otimes E\subset_\mu B\otimes E$ inside $\mathbb B(\mathcal H)\otimes E$, where $\mu=K\Big((1+\gamma)^\ell-1\Big)$.\end{corollary}

Every nuclear C$^*$-algebra has length $2$ with length constant $1$ (the similarity property for nuclear C$^*$-algebras can be found in \cite{Bunce.SimilarityNuclear}) and property $D_1$ \cite{Christensen.Perturbations2}.  In this case the $\mu$ of Proposition \ref{Length.Matrix} and Corollary \ref{Length.TensorNuclear} is given by $\mu=2\gamma+\gamma^2$ and the corollary gives better estimates than the original version \cite[Theorem 3.1]{Christensen.NearInclusions}, which uses property $D_1$ to lift near inclusions $A\subset_\gamma B$ to inclusions $A\otimes E\subset_{6\gamma}B\otimes E$, when $A$ and $E$ are nuclear.

\section{Technical Preliminaries}\label{Tech}

In this section we collect various technical results from the literature as well as establish some further  preliminaries.  We start with some standard estimates which we will use repeatedly.

\begin{proposition}\label{Prelim.Estimates}
Let $A$ and $B$ be C$^*$-subalgebras of a C$^*$-algebra $C$.
\begin{enumerate}[(i)]
\item Suppose that $A\subset_{\gamma}B$ for some $\gamma<1/2$.  Given a projection $p\in A$, there exists a projection $q\in B$ with $\|p-q\|<2\gamma$.\label{Prelim.Estimates.1}
\item Suppose that $A$ and $B$ are unital and share the same unit. Suppose that $\gamma <1$ and that $A\subset_\gamma B$. Then the following hold. 
\begin{enumerate}[(a)]
\item Given a unitary $u\in A$, there exists a unitary $v\in B$ with $\|u-v\|<\sqrt{2}\gamma$.\label{Prelim.Estimates.2}
\item Given a projection $p\in A$, there exists a projection $q\in B$ with $\|p-q\|<\gamma/\sqrt{2}$.\label{Prelim.Estimates.3}
\end{enumerate}
\item Suppose that $C$ is unital and $p,q$ are projections in $C$ with $\|p-q\|<1$. Then there is a unitary $u\in C$ with $upu^*=q$ and $\|u-I_C\|\leq\sqrt{2}\|p-q\|$.\label{Prelim.Estimates.4}
\end{enumerate}
\end{proposition}
\begin{proof}
(i)~~ This is Lemma 2.1 of \cite{Christensen.PerturbationsType1}.  Although the result in \cite{Christensen.PerturbationsType1} is stated for von Neumann algebras, the proof works for C$^*$-algebras.

\noindent (ii)~~ Both (a) and (b) are slightly weaker statements than those in \cite[Lemma 1.10]{Khoshkam.PerturbationK}. They follow from noting that the $\alpha(t)$ of \cite[1.9]{Khoshkam.PerturbationK} has $\alpha(t)\leq\sqrt{2}t$ for $0\leq t<1$.

\noindent (iii)~~This can be found as \cite[Lemma 6.2.1]{Murphy.Book}.
\end{proof}

As seen in the previous proposition, better constants are often obtained when the C$^*$-algebras we consider are both unital and share the same unit.  One way of reducing to this case is to simultaneously unitise all the algebras involved.  The next proposition offers another solution to this problem when one algebra is already unital.
\begin{proposition}\label{Unital}
Let $A$ and $B$ be C$^*$-subalgebras of a unital C$^*$-algebra $C$, and fix $\gamma$ satisfying  $d(A,B)<\gamma<1/4$.  Then $A$ is unital if and only if $B$ is unital. Furthermore, in this case there exists a unitary $u\in C$ with $\|u-1_C\|<2\sqrt{2}\gamma$ and $u1_Au^*=1_B$.
\end{proposition}
\begin{proof}
Suppose that $A$ is unital, so that its unit, $1_A$, is a projection in $C$.  By    Proposition \ref{Prelim.Estimates} (i) there exists a projection $q\in B$ with $\|1_A-q\|<2\gamma$.  We will show that $q$ is the unit of $B$.  Take $b$ in the unit ball of $B$ and find $a$ in the unit ball of $A$ with $\|a-b\|<\gamma$. Then
\begin{align}
\|qb-b\|&\leq \|q(b-a)\|+\|(q-1_A)a\|+\|a-b\|\notag\\
&\leq \gamma+2\gamma+\gamma=4\gamma<1.
\end{align}
Now let $(e_\alpha)$ be an approximate identity for $B$.  Working in $B^{**}$, we have $e_\alpha\nearrow 1_{B^{**}}$ so that taking a weak$^*$-limit in the previous estimate gives
\begin{equation}
\|1_{B^{**}}-q\|=\|1_{B^{**}}q-1_{B^{**}}\|\leq 4\gamma<1.
\end{equation}
It follows that the projection $1_{B^{**}}-q$ is zero and so $q=1_{B^{**}}$. Accordingly $B$ is unital with unit $q=1_B$.  Since $\|1_A-q\|<2\gamma<1$, Proposition \ref{Prelim.Estimates} part (\ref{Prelim.Estimates.4}) gives a unitary $u\in C$ with $\|u-1_C\|<\sqrt{2}\|1_A-1_B\|=2\sqrt{2}\gamma$ and $u1_Au^*=1_B$.
\end{proof}

Section 5 of \cite{Christensen.NearInclusions} shows that, given a sufficiently close inclusion $Q\subseteq_\gamma B$ of C$^*$-algebras with $Q$ finite dimensional, there exists a partial isometry close to $I_Q$ with $vQv^*\subseteq B$ and with all the constants independent of the structure of $Q$. When $Q$ has small dimension, better constants can be achieved using elementary techniques going back to the work of Murray and von Neumann on hyperfinite factors, subsequently employed by Glimm \cite{Glimm} and Bratteli \cite{Bratteli.AF}.  The proposition below records the constants required when $Q$ is a copy of the  $2\times 2$ matrices.  The proof is omitted.

\begin{proposition}\label{2Times2}
Let $Q,B$ be C$^*$-subalgebras of a unital C$^*$-algebra $C$ which contain $I_C$.  Suppose that $Q$ is $^*$-isomorphic to a copy of the $2\times 2$ matrices and $Q\subset_\gamma B$ for some $\gamma<1/3\sqrt{2}$. Then there exists a unitary $v\in C^*(B,Q)$ with $vQv^*\subseteq B$ and 
\begin{equation}\label{2Times2.1}
\|v-I_C\|< (3\sqrt{2}+1)\gamma.
\end{equation}
\end{proposition}

In their pioneering article \cite{Kadison.Kastler}, Kadison and Kastler showed that the type decomposition of a von Neumann algebra is stable under small perturbations. Many of our subsequent arguments use a type decomposition approach, as we handle the finite type $\mathrm{I}$, the type $\mathrm{II}_1$ and the infinite type cases separately. The constants we can achieve will depend on the constants appearing in the stability of the type decomposition.  These constants can now be improved using techniques which were not available in \cite{Kadison.Kastler}. We shall demonstrate this below in the cases we need. We will also collect some reduction arguments for later use.  Our first lemma uses results from \cite{Christensen.PerturbationsType1} and shows that, when considering close von Neumann algebras, we can reduce to the case where they have common centres.

\begin{lemma}\label{Centre}
Let $M,N\subset \mathbb B(\mathcal H)$ be von Neumann algebras whose centres are denoted $Z(M)$ and $Z(N)$ respectively.  Suppose that $d(M,N)\leq\gamma$ for some $\gamma<1/6$. Then there exists a unitary $u\in (Z(M)\cup Z(N))''$ such that $uZ(M)u^*=Z(uMu^*)=Z(N)$ and 
\begin{equation}\label{eq3.4}
\|u-I_{\mathcal H}\|\leq 2^{5/2}\gamma(1+(1-16\gamma^2)^{1/2})^{-1/2}\leq 5\gamma.
\end{equation}
In particular $d(uMu^*,N)\leq 11\gamma$ and $uMu^*$ and $N$ have common centre.
\end{lemma}
\begin{proof}
As $\gamma<1/6$, Lemma 2.2 of \cite{Christensen.PerturbationsType1} shows that the Hausdorff distance between the projections in $Z(M)$ and the projections in $Z(N)$ is at most $2\gamma$.  As $2\gamma<1/2$, the result follows from Theorem 3.2 of \cite{Christensen.PerturbationsType1}. For $\gamma< 1/6$, direct computation  gives
the second inequality of \eqref{eq3.4}.
The estimate
\begin{equation}
d(uMu^*,N)\leq d(uMu^*,M)+d(M,N)\leq 2\|u-I_{\mathcal H}\|+\gamma\leq 11\gamma
\end{equation}
follows.
\end{proof}

Once two von Neumann algebras have the same centre, we can directly compare their type decompositions. 
\begin{lemma}\label{Decomposition}
Let $M,N \subseteq\mathbb B(\mathcal H)$ be von Neumann algebras with a common centre $Z$, and suppose that $d(M,N) < 1/10$. If $z_1,z_2,z_3\in Z$ are central projections so that
\begin{equation}
M=Mz_1\oplus Mz_2 \oplus Mz_3
\end{equation}
is the decomposition of $M$ into respectively the finite type {\rm I}, type $\text{\rm II}_1$ and infinite parts, then
\begin{equation}
N=Nz_1 \oplus Nz_2 \oplus Nz_3
\end{equation}
is the corresponding decomposition for $N$.
\end{lemma}

\begin{proof}
Let $N=N\tilde z_1 \oplus N\tilde z_2 \oplus N\tilde z_3$ be the corresponding decomposition for $N$. We first show that $z_3=\tilde z_3$. If this is not the case then, without loss of generality, there is a non-zero central projection $z$ such that $Mz$ is finite and $Nz$ is infinite. By cutting by $z$, we may then assume that $M$ is finite and $N$ is infinite. Let $v\in N$ be an isometry which is not a unitary, and choose $x\in M$ with $\|x-v\|<10^{-1}$ and $\|x\|\le 1$. For each $\xi\in\mathcal H$,
\begin{equation}\label{eq40.1}
 (1-10^{-1})\|\xi\| \le \|x\xi\| \le \|\xi\|
\end{equation}
and so
\begin{equation}\label{eq40.2}
 (1-10^{-1})^2 I \le x^*x \le I.
\end{equation}
Thus $|x|$ is invertible, so $u = x|x|^{-1} \in M$ satisfies $u^*u=I$. By finiteness of $M$, $u$ is a unitary, and so $x$ is invertible with $\|x^{-1}\| \le (1-10^{-1})^{-1}$ from \eqref{eq40.1}. Then
\begin{equation}\label{eq40.3}
 \|I-x^{-1}v\| = \|x^{-1}(x-v)\| \le (1-10^{-1})^{-1} 10^{-1} < 1,
\end{equation}
showing that $x^{-1}v$ and hence $v$ are invertible. This contradicts the assumption that $v$ is not a unitary and establishes that $z_3=\tilde z_3$.

After cutting by $(I-z_3)$ we may now assume that both $M$ and $N$ are direct sums of finite type I parts and type $\text{II}_1$ parts, so that $z_1+z_2 = \tilde z_1+\tilde z_2 = I$. To establish that $z_1=\tilde z_1$ we again argue by contradiction by assuming that there is a central projection $z$ so that $Mz$ is finite type $\mathrm{I}$ and $Nz$ is type $\text{II}_1$, and after cutting by $z$ we can make these assumptions on $M$ and $N$. Let $p\in M$ be a non-zero abelian projection and choose,
by     Proposition \ref{Prelim.Estimates} (ii b), a projection $q\in N$ with $\|p-q\| < 1/(10\sqrt 2)$.  Thus $d(pMp,qNq)\leq d(M,N)+2\|p-q\|\leq (1+\sqrt{2})/10<1/4$.  By \cite[Lemma 2.3]{Christensen.PerturbationsType1}, $qNq$ is abelian and so $q\in N$ is a non-zero abelian projection. This contradiction proves the result.
\end{proof}

Given a finite von Neumann algebra $M$, we write $\ct_M$ for the centre valued trace on $M$.  The next lemma examines the behaviour of centre valued traces on close projections.  We need it both in Section \ref{Length} for our analysis of C$^*$-algebras close to those of finite length and in Section \ref{KTrace} to examine traces of close C$^*$-algebras.
The next result and some succeeding ones are phrased in terms of near containments rather than distances in order to obtain better estimates.
\begin{lemma}\label{CentralTrace}
Let $M$ and $N$ be finite von Neumann algebras acting on a Hilbert space $\mathcal H$ with common centre $Z=Z(M)=Z(N)$. Suppose that $M\subseteq_\gamma N$ and $N\subseteq_\gamma M$ for some constant $\gamma<1/200$. If $p\in M$ and $q\in N$ are projections  with $\|p-q\|<1/2$, then $\ct_M(p)=\ct_N(q)$.
\end{lemma}

\begin{proof}
By Lemma \ref{Decomposition}, there is a central projection $z$ such that $Mz$ and $Nz$ are  finite and of type $\mathrm{I}$ while $M(1-z)$ and $N(1-z)$ are type $\text{II}_1$. It suffices to consider these parts separately, so we initially assume that $M$ and $N$ are finite type $\mathrm{I}$, and thus injective.

Since  both algebras are injective, the bound on $\gamma$ allows us to apply  
\cite[Corollary 4.4]{Christensen.NearInclusions} to conclude that there is a surjective isomorphism $\phi:M\rightarrow N$ satisfying $\|\phi(x)-x\| \le 100\gamma\|x\|< (1/2)\|x\|$. Accordingly
\begin{equation}\label{eq40.13}
\|\phi(p)-q\|< 100/200 + 1/2 = 1
\end{equation}
so $\phi(p)$ and $q$ are equivalent projections in $N$. Thus $\ct_N(\phi(p))=\ct_N(q)$.
Now $\phi$ maps $Z$ to $Z$ and also fixes the elements of $Z$ pointwise because central projections $z\in Z$ satisfy $\|\phi(z)-z\| \le 1/2$. Thus 
\begin{equation}
\ct_M(x)=\phi(\ct_M(x))=\ct_{\phi(M)}(\phi(x))=\ct_N(\phi(x)),\quad x\in M.
\end{equation}
Then $\ct_M(p)=\ct_N(\phi(p))$ and the result is proved in this case.

Now assume that $M$ and $N$ are both type $\text{II}_1$ and, to derive a contradiction, suppose that $\ct_M(p) \ne\ct_N(q)$. By cutting by a suitable central projection, we may assume without loss of generality that there exist constants $0\leq c<d$ such that 
\begin{equation}\label{eq40.14}
\ct_N(q) \le cI < dI \leq\ct_M(p).
\end{equation}
Choose an integer $n$ satisfying $1/n < d-c$. Then $[d,1]$ is covered by the collection of intervals $[j/n, (j+1)/n)$, $1\le j\le n$, so the spectral projections of $\ct_M(p)$ for these intervals cannot all be 0. Choose one that is non-zero and cut by this central projection. This allows us to make the further assumption that
\begin{equation}\label{eq40.15}
(j/n)I \leq\ct_M(p) < ((j+1)/n)I
\end{equation}
for some integer $j\in \{1,2,\ldots, n\}$. The case $j=n$ implies that $p=1$, whereupon $q=1$ follows from $\|p-q\| < 1/2$, and a contradiction is reached. Thus we can assume $j<n$. We may then choose orthogonal projections $e_1,\ldots, e_j\in M$ satisfying $e_i\leq p$ and $\ct_M(e_i)=I/n$, $1\le i\le j$. Since $\ct_{M}(I-p) > ((n-j-1)/n)I$, we may also choose orthogonal projections $f_i\in M$, $1\le i \le n-j-1$, satisfying $\ct_M(f_i) = I/n$ and $f_i\le I-p$. Note that there may be no $f_i$'s if $j=n-1$. Let $h = I - \sum^j_{i=1} e_i - \sum^{n-j-1}_{i=1} f_i$, which also has centre valued trace $I/n$. Then $\{e_1,\ldots, e_j, h, f_1,\ldots, f_{n-j-1}\}$ is a set of $n$ equivalent projections in $M$ with sum $I$, so lie in a matrix subalgebra $F\subset M$ as the minimal diagonal projections. Let $h_1 = p - \sum^j_{i=1}e_i$ and $h_2 = (1-p) - \sum^{n-j-1}_{i=1} f_i$. Then $h_1+h_2=h$ and $h_1,h_2\le h$. Thus the algebra $Q$ generated by $h_1,h_2$ and $F$ is finite dimensional, so injective, and contains $p$. By \cite[Theorem 4.3]{Christensen.NearInclusions} there is a $*$-isomorphism $\phi$ of $Q$ into $N$ satisfying $\|\phi(x) - x\| \le 100\gamma\|x\|$, for $x\in Q$. Again $\|\phi(p)-q\| < 100/200 + 1/2 = 1$, so $\phi(p)$ and $q$ are equivalent in $N$ which yields $\ct_N(\phi(p))=\ct_N(q)$. The projections $\{\phi(e_1),\ldots, \phi(e_j), \phi(h), \phi(f_1),\ldots, \phi(f_{n-j-1})\}$ are equivalent in $N$ and sum to $I$. Thus each has centre valued trace $I/n$. It follows that
\begin{equation}
\ct_N(q)=\ct_N(\phi(p))\geq \sum^j_{i=1}\ct_N(\phi(e_i))\geq j/nI>\ct_M(p)-1/nI\geq (d-1/n)I.
\end{equation}
This implies $d-c\leq 1/n$, contradicting the choice of $n$, and proving the result.
\end{proof}

The next result in this section  examines von Neumann algebras close to those in standard position. Recall from \cite[I {\S}6.1]{Dixmier.AlgOpsHilbertFrench} or \cite[p. 691]{KR.2} that the coupling function $\Gamma(M,M')$ for a finite von~Neumann algebra $M$ with finite commutant $M'$ is a possibly unbounded positive operator affiliated to the centre $Z$, having the following property. For each vector $\xi$ in the underlying Hilbert space $\mathcal H$
\begin{equation}\label{eq4.1}
 \mathbb T_M(e^{M'}_\xi) = \Gamma(M,M') \mathbb T_{M'}(e^{M}_\xi)
\end{equation}
where $e^{M'}_\xi\in {M}$ is the projection onto the cyclic subspace $\overline{M'\xi}$, while $e^{M}_\xi \in M'$ projects onto $\overline{M\xi}$. Recall too that a finite von Neumann algebra $M$ is in standard position on a Hilbert space $\mathcal H$ if and only if $M'$ is finite and $\Gamma(M,M')=I$.  From this point of view, the next lemma shows that a von~Neumann algebra which is close to an algebra in standard position is approximately in standard position.

\begin{lemma}\label{CouplingConstant}
Let $M$ and $N$ be finite von~Neumann algebras on a Hilbert space $\mathcal H$ with common centre $Z$. Let $ \gamma,\,\delta < 1/200$ be constants such that the near inclusions
\begin{equation}
M\subseteq_\gamma N,\quad N\subseteq_\gamma M,\quad M'\subseteq_\delta N',\quad and\quad N'\subseteq_\delta M' 
\end{equation}
hold.
 If $M$ is in standard position on $\mathcal H$, then $N'$ is finite and $\Gamma(N,N')$ satisfies
\begin{equation}\label{eq4.2}
 0.99~I <  (1-\gamma/\sqrt{2})I \le \Gamma(N,N') \le \frac{1}{1-\delta/\sqrt{2}} I < 1.01~I.
\end{equation}
\end{lemma}

\begin{proof}
Since $d(M',N')\le 2\delta<1/100<1/10$ and $M'$ is finite, Lemma \ref{Decomposition} shows that $N'$ is also finite. We are  not requiring that $\mathcal H$ be separable, so $M$ need not have a faithful trace.  However, as $M$ certainly has a separating family of normal tracial states, a maximality argument gives a set $\{z_j\}_{j\in J}$ of orthogonal central projections summing to $I$ so that each $Mz_j$ has a faithful normal trace.  By proving the result for each of the pairs $(Mz_j,Nz_j)$ separately, we may then assume that $M$ has a faithful normal trace $\tau$. Let $t>0$ be fixed but arbitrary in the spectrum of $\Gamma(N,N')$. It suffices to demonstrate the inequalities
\begin{equation}\label{eq4.3}
1-\gamma/\sqrt{2} \le t\le \frac{1}{1-\delta/\sqrt{2}}.
\end{equation}

Given $\varepsilon>0$, let $e\in Z$ be the non-zero spectral projection of $\Gamma(N,N')$ for $(t-\varepsilon,t+\varepsilon)$. We may cut by this projection, which allows us to assume that
\begin{equation}\label{eq4.4}
 (t-\varepsilon) I \le \Gamma(N,N') \le (t+\varepsilon)I.
\end{equation}
Since $M$ is in standard position on $\mathcal H$, there is a unit vector $\xi\in\mathcal H$ so that the vector state $\langle\,\cdot\,\xi,\xi\rangle$ defines a faithful tracial state  both on $M$ and on $M'$. Define two cyclic projections $p\in N$ and $q\in N'$ with range spaces $\overline{N'\xi}$ and $\overline{N\xi}$ respectively. By   Proposition \ref{Prelim.Estimates} (ii b), we may choose projections $r\in M$ and $s\in M'$ so that $\|p-r\| \le \gamma/\sqrt 2$ and $\|q-s\| \le \delta/\sqrt 2$. The hypotheses of Lemma \ref{CentralTrace} are satisfied and so $\ct_{N}(p)=\ct_M(r)$ and $\ct_{N'}(q)=\ct_{M'}(s)$.
The centre valued traces $\ct_M$ and $\ct_{M'}$ preserve the trace $\langle\,\cdot\,\xi,\xi\rangle$ on $M$ and $M'$ and so
\begin{equation}\label{eq4.5}
 \langle \ct_{N'}(q)\xi,\xi\rangle = \langle\ct_{M'}(s)\xi,\xi\rangle = \langle s\xi,\xi\rangle
\end{equation}
and
\begin{equation}\label{eq4.6}
 \langle\ct_N(p)\xi,\xi\rangle = \langle\ct_M(r)\xi,\xi\rangle = \langle r\xi,\xi\rangle.
\end{equation}
Define $\alpha$ to be $\langle \ct_{N'}(q)\xi,\xi\rangle > 0$, and $\beta$ to be such that
\begin{equation}\label{eq4.7}
 \alpha\beta = \langle\ct_N(p)\xi,\xi\rangle = \langle \Gamma(N,N')\ct_{N'}(q)\xi,\xi\rangle.
\end{equation}
The relations \eqref{eq4.4} and \eqref{eq4.7} imply that
\begin{equation}\label{eq4.8}
 (t-\varepsilon)\alpha \le \alpha\beta \le (t+\varepsilon)\alpha
\end{equation}
and so
\begin{equation}\label{eq4.9}
 \beta-\varepsilon \le t\le \beta+\varepsilon.
\end{equation}
Since $p\xi = q\xi = \xi$, the choices of $r$ and $s$ imply that
\begin{equation}\label{eq4.10}
 1-\delta/\sqrt 2 \le \langle s\xi,\xi\rangle \le 1
\end{equation}
and
\begin{equation}\label{eq4.11}
 1-\gamma/\sqrt 2 \le \langle r\xi,\xi\rangle \le 1.
\end{equation}
The definitions of $\alpha$ and $\alpha\beta$, together with (\ref{eq4.5}) and (\ref{eq4.6}), allow us to rewrite these inequalities as
\begin{equation}\label{eq4.12}
 1-\delta/\sqrt 2 \le \alpha \le 1
\end{equation}
and
\begin{equation}\label{eq4.13}
 1-\gamma/\sqrt 2 \le \alpha\beta \le 1,
\end{equation}
after which division yields
\begin{equation}\label{eq4.14}
 1-\gamma/\sqrt{ 2} \le \beta \le \frac{1}{1-\delta/\sqrt{2}}.
\end{equation}
From \eqref{eq4.9}, we now have the inequalities
\begin{equation}\label{eq4.15}
 1-\gamma/\sqrt{ 2} - \varepsilon \le t\le \frac{1}{1-\delta/\sqrt{2}} +\varepsilon.
\end{equation}
Now let $\varepsilon\to 0$, and we have proved \eqref{eq4.3} as required.
\end{proof}

We end the section with a final technical result which we need in the proof of Lemma \ref{Length.II1Case}.
\begin{lemma}\label{Matrix.Units}
Let $N\subseteq \mathbb{B}(\mathcal{H})$ be a von Neumann
algebra and let $d\in N'$ be a projection with central
support $I_{\mathcal H}$. Then there exists an infinite
dimensional Hilbert space $\mathcal G$ and a minimal
projection $g_0\in \mathbb{B}(\mathcal{G})$ such that
$d\otimes g_0$ extends to a system of matrix units in
$N'\ \overline{\otimes}\ \mathbb{B}(\mathcal{G})$.
\end{lemma}

\begin{proof}
Since $d$ has full central support, the map $n\mapsto nd$
defines an isomorphism between $N$ and $Nd\subseteq
\mathbb{B}(d(\mathcal{H}))$. The general theory of
isomorphisms, \cite[I {\S}4 Theorem 3]{Dixmier.AlgOpsHilbertFrench}, gives an infinite dimensional
Hilbert space $\mathcal G$ so that the amplifications
$N\otimes I_{\mathcal G}$ and $Nd\otimes I_{\mathcal G}$ are
spatially isomorphic by a unitary $v:{\mathcal H}\otimes
{\mathcal G} \to d({\mathcal H})\otimes {\mathcal G}$. We
regard this operator as a partial isometry on $\mathcal H\otimes\mathcal G$ with initial projection $I_{\mathcal H}\otimes I_{\mathcal G}$ and final projection $d\otimes I_{\mathcal G}$. Multiplying the following equation 
\begin{equation}\label{neweq}
v(n\otimes I_{\mathcal G})v^*=(nd)\otimes I_{\mathcal
G},\qquad n\in N,
\end{equation}
on the left by $v^*$ gives
\begin{equation}
(n\otimes I_{\mathcal G})v^*=
v^*(d\otimes I_{\mathcal G})(n\otimes I_{\mathcal G})=
v^*(n\otimes I_{\mathcal G}),\qquad n\in N,
\end{equation}
from which we conclude that $v\in
N'\ \overline{\otimes}\ \mathbb{B}(\mathcal{G})$.

Now split $\mathcal G$ as $\mathcal{E}\oplus \mathcal{F}$,
where these summands have the same infinite dimension as $\mathcal
G$, and define $q\in N'\ \overline{\otimes}\
\mathbb{B}(\mathcal{G})$ to be $d\otimes I_{\mathcal E}$
which is equivalent to $d\otimes I_{\mathcal F}$ and $d\otimes I_{\mathcal G}$. Then
define $p=(I_{\mathcal H}\otimes I_{\mathcal G})-q \geq
d\otimes I_{\mathcal F}$. The latter projection is
equivalent to $d\otimes I_{\mathcal G}$ and so $I_{\mathcal H}\otimes I_{\mathcal G}\sim d\otimes I_{\mathcal F}\leq p\precsim I_{\mathcal H}\otimes I_{\mathcal G}$. Thus $p$, $d\otimes I_{\mathcal F}$, $I_{\mathcal H}\otimes I_{\mathcal G}$ and $q=d\otimes I_{\mathcal E}$ are all equivalent in $N'\ \overline{\otimes}\ \mathbb B(\mathcal G)$. Choose a partial isometry
$w\in N'\ \overline{\otimes}\ \mathbb{B}(\mathcal{G})$ so that
$w^*w=p=(I_{\mathcal H}\otimes I_{\mathcal G})-q$ and
$ww^*=d\otimes I_{\mathcal E}=q$.
Now choose a family of orthogonal equivalent
projections $\{e_j: j\in J\}\subseteq
\mathbb{B}(\mathcal{E})$ with sum $I_{\mathcal E}$. Then the
equivalent projections $\{d\otimes e_j: j\in J\}$ sum to
$q$, and these in turn are equivalent to $\{w^*(d\otimes
e_j)w: j\in J\}$ with sum $(I_{\mathcal H}\otimes
I_{\mathcal G})-q$. The proof is completed by choosing $g_0$
to be any one of the projections $e_j$.
\end{proof}

\section{Stability of finite length}\label{Length}
The main result of this section is Theorem \ref{Length.DH}  which shows that a C$^*$-algebra $B$ which is close to a C$^*$-algebra $A$ of finite length must also have finite length and obtains a bound on the length of $B$ in terms of the length and the length constant of $A$.  When $A$ has finite length, Proposition \ref{LDK.Comm} gives a near inclusion of $B'$ inside $A'$ (with constants depending on $d(A,B)$,  the length of $A$, and its associated length constant). The key step in Theorem \ref{Length.DH} is to obtain a reverse near inclusion of $A'$ inside $B'$ which we achieve in Theorem \ref{Length.Comm}.  This in turn is established by a type decomposition argument, handling the finite type $\mathrm{I}$, the type $\mathrm{II}_1$, and the infinite cases separately. Existing results enable us to deal with the first and last cases quickly so the heart of the matter is the $\mathrm{II}_1$ case.

\begin{lemma}\label{Length.II1Case}
Let $M$ and $N$ be von Neumann algebras of type $\mathrm{II}_1$ faithfully and non-degenerately represented on $\mathcal H$. Suppose further that $M$ and $N$ have common centre $Z$ which admits a faithful state. Suppose that $d(M,N)=\alpha$ and $M$ contains an ultraweakly dense C$^*$-algebra $A$ of length at most $\ell$ and length constant at most $K$. Write $k=K\ell/2$.  If $\alpha$ satisfies the inequality
\begin{equation}\label{Length.II1Case.Est}
24(12\sqrt{2}k+4k+1)\alpha<1/200,
\end{equation}
then
\begin{equation}
d(M',N')\leq2\beta+1200k\alpha(1+\beta),
\end{equation}
where $\beta=K((1+28800k\alpha+48\alpha)^\ell-1)$.
\end{lemma}

The proof of this result is long and intricate, so it will
be helpful to give a brief summary before embarking on it.
Our objective is to reduce to the following situation: 
\begin{itemize}
\item[{\rm (i)}] $\mathcal H$
decomposes as $\mathcal H_0\otimes \ell^2(\Lambda)$; 
\item[{\rm (ii)}]
the von Neumann
algebras $M$ and $N$ simultaneously decompose as $M\cong M_0\otimes
I_{\ell^2(\Lambda)}$ and $N\cong N_0\otimes I_{\ell^2(\Lambda)}$;
\item[{\rm (iii)}] 
$M_0$ is in standard position on $\mathcal H_0$; 
\item[{\rm (iv)}] 
$N_0$ has the
local
distance property $LD_{24}$ on $\mathcal H_0$.
\end{itemize}
Once (i)-(iv) have been achieved, the proof is completed by the following steps.
  The local distance
property immediately gives a near inclusion
\begin{equation}\label{Para}
M_0'\subseteq_{\alpha'}N_0',\quad\text{on }\mathcal H_0
\end{equation}
for a suitable constant $\alpha'$.
Since $M_0$ is in standard position on $\mathcal H_0$, it is
anti-isomorphic
to its commutant. In particular, this commutant has a weakly dense
C$^*$-algebra of finite length so we can use results from Section \ref{Sim}
to lift the near inclusion (\ref{Para}) to obtain a near inclusion
of
the form
\begin{equation}
M'=M_0'\ \overline{\otimes}\ \mathbb
B(\ell^2(\Lambda))\subseteq_{\alpha''}N_0'\ \overline{\otimes}\ \mathbb
B(\ell^2(\Lambda))=N'
\end{equation}
for a suitable constant $\alpha''$. Since a reverse near inclusion is immediate from the hypotheses of
the
lemma, this establishes the result.

To reach the situation detailed in (i)-(iv) above, a number of further
reductions
are necessary.  We first adjust the Hilbert space and arrange for
the
representation of $N$ to be an amplification of its standard position.
We
then find non-zero close projections $e\in M'$ and $d\in N'$ of full
central support so that $Me$ is in standard position on 
$e(\mathcal H)$ and $Nd$ has the local distance property $LD_{24}$ on 
$d(\mathcal H)$.  This is the main technical step in the proof, requiring our
earlier results regarding the behaviour of the centre valued trace
and coupling function under small perturbations. This enables us to transfer the
property that some cut down of $N$ is in standard position to the
same
property for $M$.  We then use the perturbation theory for injective
von Neumann algebras from \cite{Christensen.NearInclusions} to
obtain
the situation of (i)-(iv) above.  The authors would like to thank the referee for bringing a small gap in the original version of this lemma to our attention.

\begin{proof}[Proof of Lemma \ref{Length.II1Case}]
Let $S$ be an isomorphic copy of $N$, acting in standard position on a Hilbert space $\mathcal K$. The general theory of isomorphisms of von~Neumann algebras \cite[I {\S}4 Theorem 3]{Dixmier.AlgOpsHilbertFrench} allows us to choose a sufficiently large set $\Omega$ (which we insist has at least 2 points) so that the amplifications $\widetilde{N}$ of $N$ to $\mathcal H\otimes \ell^2(\Omega)$ and $\widetilde{S}$ of $S$ to $\mathcal K\otimes \ell^2(\Omega)$ are spatially isomorphic. Amplification increases the distance between commutants, so if the result is true in this context then it is true generally. Thus we can assume that $\mathcal H$ decomposes as $\mathcal K\otimes \ell^2(\Omega)$ and that $N=S \otimes I_{\ell^2(\Omega)}$. Then $N'=S'\ \overline\otimes\ \mathbb B(\ell^2(\Omega))$. 

Proposition \ref{Prelim.LengthDK} shows that the C$^*$-algebra $A$ has property $D_k$, so
Proposition \ref{LDK.Comm} gives
\begin{equation}\label{II1.1}
N'\subseteq_{2k\alpha}M'.
\end{equation}
Choose a copy $Q_0$ of the $2\times 2$ matrices in $\mathbb B(\ell^2(\Omega))$ such that the minimal projections of $Q_0$ are rank one projections in $\mathbb B(\ell^2(\Omega))$ and let $Q=I_{\mathcal K}\otimes Q_0\subset\ N'$. The near inclusion (\ref{II1.1}) gives
\begin{equation}
Q\subseteq_{2k\alpha}M',
\end{equation}
and note that $Q$ and $M'$ both lie in the algebra $Z'$. The inequality (\ref{Length.II1Case.Est}) implies $2k\alpha<1/(3\sqrt{2})$, so Proposition \ref{2Times2} gives us a unitary $u_1\in Z'$ with 
\begin{equation}\label{II1.2}
\|u_1-I_{\mathcal K}\|\leq (3\sqrt{2}+1)2k\alpha,
\end{equation}
such that $u_1Qu_1^*\subset M'$.  

Define $N_1=u_1Nu_1^*$. Since $u_1\in Z'$ it follows that $N_1$ has centre $Z$.  Let $Q_1=u_1Qu_1^*$ so that $Q_1\subset M'\cap N_1'$. The estimate (\ref{II1.2}) gives the distance estimate
\begin{equation}\label{L.3}
d(M,N_1)\leq 2\|u_1-I_{\mathcal K}\|+d(M,N)\leq (12\sqrt{2}k+4k+1)\alpha.
\end{equation}
Similarly, the near inclusion (\ref{II1.1}) induces the near inclusion
\begin{equation}\label{L.4}
N_1'\subseteq_{6(2\sqrt{2}+1)k\alpha}M'.
\end{equation}
The construction of $Q$ ensures that every non-zero projection in $Q_1$ has central support $I$ in $N_1'$ and hence central support $I$ in $M'$.  Fix a minimal projection $f\in Q_1$. By choice of $Q_1$, the algebra $N_1f$ is in standard position on $f(\mathcal H)$. Then $N_1f$ has a cyclic vector and so has the local distance property $LD_{12}$ on this space by Proposition \ref{LDK.Cyclic}.  The distance estimate (\ref{L.3}) compresses to $f(\mathcal H)$ to give the near inclusion
\begin{equation}
N_1f\subseteq_{(12\sqrt{2}k+4k+1)\alpha}Mf.
\end{equation}
Applying Proposition \ref{LDK.Comm} then gives
\begin{equation}
(M')_f\subseteq_{24(12\sqrt{2}k+4k+1)\alpha}(N_1')_f.
\end{equation}
Since $f$ lies in $M'\cap N_1'$, we can also compress (\ref{L.4}) by $f$ to obtain
\begin{equation}
(N_1')_f\subseteq_{6(2\sqrt{2}+1)k\alpha} (M')_f.
\end{equation}

Now $N_1f$ is in standard position on $f(\mathcal H)$. The inequalities (\ref{Length.II1Case.Est})
and (\ref{L.3}) ensure that $d(Mf,N_1f)<1/200$.   Moreover,
\begin{equation}
6(\sqrt{2}+1)k\alpha<1/200\quad\text{and}\quad 24(12\sqrt{2}k+4k+1)\alpha<1/200,
\end{equation}
so  the hypotheses of Lemma \ref{CouplingConstant} are met for the algebras $Mf$ and $N_1f$. Writing $\Gamma(Mf,(M')_f)$ for the coupling function of $Mf$ on $f(\mathcal H)$, we obtain
\begin{equation}
0.99~f\leq \Gamma(Mf,(M')_f)\leq 1.01~f.
\end{equation}
Let $I_{Q_1}$ denote the unit of $Q_1$ and $\Gamma(MI_{Q_1},(M')_{I_{Q_1}})$ denote the coupling function of $MI_{Q_1}$ on $I_{Q_1}(\mathcal H)$. As $MI_{Q_1}$ is a two-fold amplification of $Mf$, it follows that
\begin{equation}
1.98~I_{Q_1}\leq \Gamma(MI_{Q_1},(M')_{I_{Q_1}})\leq 2.02~I_{Q_1}.
\end{equation}
In particular $\Gamma(MI_{Q_1},(M')_{I_{Q_1}})\geq I_{Q_1}$ and so any state on $MI_{Q_1}$ is a vector state (see \cite[III {\S}.1 Proposition 3]{Dixmier.AlgOpsHilbertFrench}, for example).

Let $\tau$ be a faithful tracial state on $M$, the existence of which is guaranteed by our hypothesis that $Z$ admits a faithful state.  As $I_{Q_1}$ has central support $I$ in $M'$, the representation $m\mapsto mI_{Q_1}$ of $M$ on $I_{Q_1}(\mathcal H)$ is faithful. Therefore the previous paragraph gives us a unit vector $\xi\in I_{Q_1}(\mathcal H)$ with
\begin{equation}
\tau(m)=\langle m\xi,\xi\rangle,\quad m\in M.
\end{equation}
Let $e_0\in M'$ be the projection onto $\overline{M\xi}$. Then $Me_0$ is in standard position on $e_0(\mathcal H)$ and $e_0\leq I_{Q_1}$. Since the range of $e_0$ contains a trace vector for 
the faithful trace $\tau$ on $M$, it follows that $e_0$ has central support $I$ for $M'$. Indeed, given a non-zero projection $z\in Z$, $\tau(z)=\langle z\xi,\xi\rangle=\langle ze_0\xi,\xi\rangle \neq 0$, so that $ze_0\neq 0$.

By construction $(N_1)_{I_{Q_1}}$ has a $2$-cyclic set and so property $LD_{24}$ by Proposition \ref{LDK.Cyclic}. Accordingly, putting the distance estimate (\ref{L.3}) into Proposition \ref{LDK.Comm} gives the near inclusion
\begin{equation}
(M')_{I_{Q_1}}\subset_{48(12\sqrt{2}k+4k+1)\alpha}(N_1')_{I_{Q_1}}.
\end{equation}
 Proposition \ref{Prelim.Estimates} (ii b) then allows us to find a projection $d_0\in (N_1')_{I_{Q_1}}$ with
\begin{equation}
\|e_0-d_0\|\leq 48(12\sqrt{2}k+4k+1)\alpha/\sqrt{2}.
\end{equation}
Since $N_1I_{Q_1}$ has a $2$-cyclic set, so too does $N_1d_0$.  In particular Proposition \ref{LDK.Cyclic} shows that the algebra $N_1d_0$ on $d_0(\mathcal H)$ retains property $LD_{24}$. Since $\|e_0-d_0\|<1$ (this follows from the inequality (\ref{Length.II1Case.Est})) and $M'$ and $N_1'$ have common centres, $d_0$ has central support $I$ in $N_1'$.

Define $d=u_1^*d_0u_1$.  This lies in $N'$ and has the same properties there that $d_0$ has in $N_1'$. Thus the algebra  $Nd$ on $d(\mathcal H)$ has the local distance property $LD_{24}$, $d$ has central support $I$ in $N'$ and $d$ is finite in $N'$. It is convenient to adjust $e_0$ as this improves the estimates obtained in the lemma.  Since $N'\subset_{2k\alpha}M'$, applying  Proposition \ref{Prelim.Estimates} (ii b) again gives us a projection $e\in M'$ with
\begin{equation}
\|d-e\|<2k\alpha/\sqrt{2}=\sqrt{2}k\alpha.
\end{equation}
It follows that 
\begin{align}
\|e-e_0\|\leq &\|e-d\|+\|d-d_0\|+\|e_0-d_0\|\\
\leq&2k\alpha/\sqrt{2}+2\|u_1-1_{\mathcal H}\|+48(12\sqrt{2}k+4k+1)\alpha/\sqrt{2}<1,
\end{align}
where we obtain the bound of $1$ from the inequality (\ref{Length.II1Case.Est}).
Then $e$ and $e_0$ are unitarily equivalent in $M'$ and in particular $e$ has central support $I$ and $Me$ is in standard position on $e(\mathcal H)$.  This completes the first stage of the proof. 

Lemma \ref{Matrix.Units} enables us to find a Hilbert space $\mathcal G$ and a minimal projection $g_0\in\mathbb B(\mathcal G)$ so that there is a family of matrix units $(d_{i,j})_{i,j\in\Lambda}$ in $N'\ \overline{\otimes}\ \mathbb B(\mathcal G)$ with $d_{i_0,i_0}=d\otimes g_0$. As at the beginning of the proof, amplification increases the distance between commutants.  Thus we can work on $\mathcal H\otimes\mathcal G$, replacing $M$ and $N$ by $M\otimes I_{\mathcal G}$ and $N\otimes I_{\mathcal G}$ respectively. Note that (\ref{II1.1}) still holds as (the amplified versions of) $A$ and $N$ have property $D_k$.   Let $P$ be the injective von Neumann subalgebra of $N'\subseteq Z'$ generated by the matrix units $(d_{i,j})_{i,j\in\Lambda}$.  The near inclusion (\ref{II1.1}) gives
\begin{equation}
P\subset_{2k\alpha}M',
\end{equation}
so by \cite[Theorem 4.3]{Christensen.NearInclusions}, there is a unitary $u_2\in (M'\cup N')''\subseteq Z'$ such that $\|u_2-I_{\mathcal H}\|\leq 300k\alpha$ and $u_2Pu_2^*\subset M'$.  Again the required hypothesis that $2k\alpha<1/100$ to use \cite[Theorem 4.3]{Christensen.NearInclusions} is immediate from our initial inequality (\ref{Length.II1Case.Est}).

Define $N_2=u_2Nu_2^*$. This algebra also has centre $Z$ as $u_2\in Z'$. Define matrix units by $e_{i,j}=u_2d_{i,j}u_2^*$ and note that these matrix units lie in $M'\cap N_2'$.  The projection $e_{i_0,i_0}$ has 
\begin{equation}
\|e_{i_0,i_0}-e\otimes g_0\|\leq\|e_{i_0,i_0}-d_{i_0,i_0}\|+\|e\otimes g_0-d\otimes g_0\|\leq 2\|u_2-I_{\mathcal H}\|+\|e-d\|<1,
\end{equation}
where again we collect our previous estimates and apply (\ref{Length.II1Case.Est}) to achieve this estimate. Therefore $e_{i_0,i_0}$ and $e\otimes g_0$ are unitarily equivalent in $M'$ and so $Me_{i_0,i_0}$ is in standard position on $e_{i_0,i_0}(\mathcal H\otimes\mathcal G)$.  Using these matrix units we see that $M'$ and $N_2'$ are simultaneously spatially isomorphic to $(M')_{e_{i_0,i_0}}\ \overline{\otimes}\ \mathbb B(\ell^2(\Lambda))$ and $(N_2')_{e_{i_0,i_0}}\ \overline{\otimes}\ \mathbb B(\ell^2(\Lambda))$. 
The algebras $M$ and $N_2$ are now in the position described by conditions (i)-(iv) in the discussion preceding the proof. 
To ease notation write $T_M$ for the von Neumann algebra $Me_{i_0,i_0}$ acting on $e_{i_0,i_0}(\mathcal H\otimes\mathcal G)=\mathcal H_0$ and $T_{N_2}$ for $N_2e_{i_0,i_0}$ acting on the same space.  We have the distance estimate
\begin{equation}\label{L.10}
d(T_M,T_{N_2})\leq d(M,N_2)\leq d(M,N)+2\|u_2-I_{\mathcal H}\|\leq 600k\alpha+\alpha.
\end{equation}
By construction $T_{N_2}=N_2e_{i_0,i_0}$ has property $LD_{24}$ on $\mathcal H_0$ so Proposition \ref{LDK.Comm} gives
\begin{equation}
T_M'\subset_{48(600k\alpha+\alpha)}T_{N_2}'.
\end{equation}

Since $T_M$ lies in standard position, there is a conjugate linear isometry $J$ on $\mathcal H_0=e_{i_0,i_0}(\mathcal H\otimes \mathcal G)$ with $JT_MJ=T_M'$.  Now $T_M$, as a cut down of $M$, has a weak$^*$-dense C$^*$-algebra with length at most $\ell$ with length constant at most $K$.  Write $T_A$ for this C$^*$-algebra and note that $JT_AJ$ is weak$^*$-dense in $T_M'$ and also has length at most $\ell$ and length constant at most $K$.  
Since
\begin{equation}
JT_AJ\subset_{48(600k\alpha+\alpha)} T_{N_2}',
\end{equation}
Corollary \ref{Length.TensorNuclear} gives
\begin{equation}
JT_AJ\otimes\mathbb K(\ell^2(\Lambda))\subset_{\beta}
T_{N_2}'\otimes\mathbb K(\ell^2(\Lambda)),
\end{equation}
where $\beta=K((1+28800k\alpha+48\alpha)^\ell-1)$. Lemma 5 of \cite{Kadison.Kastler} allows us to take the weak operator closure of this near inclusion (note that although the statement is only given for the two-sided notion of closeness, the proof works in the one-sided context we need).  This gives a near inclusion
\begin{equation}
M'\subset_\beta N_2'.
\end{equation}
Since $d(N_2',N')\leq 2\|u_2-I\|\leq 600k\alpha$,  Proposition \ref{Prelim.MetricNear} (i) gives 
\begin{equation}
M'\subset_{\beta+600k\alpha(1+\beta)}N'.
\end{equation}
Combining this with the initial near inclusion (\ref{II1.1}) and using  Proposition \ref{Prelim.MetricNear} (iii) gives the estimate
\begin{equation}
d(M',N')\leq 2\beta+1200k\alpha(1+\beta),
\end{equation}
which completes the proof.
\end{proof}

The next theorem combines the previous lemma with results from \cite{Christensen.Perturbations2} to show that sufficiently close algebras have close commutants if one algebra has finite length. We do not assume that $A$ and $B$ are represented non-degenerately and so we use the notation $\overline{A}^w$ rather than $A''$ to denote the von Neumann algebra generated by $A$.

\begin{theorem}\label{Length.Comm}
Let $A$ and $B$ be $C^*$-algebras acting on a Hilbert space $\mathcal H$. Let $\gamma$ denote $d(A,B)$. Suppose that $A$ has finite length at most $\ell$ with length constant at most $K$, and suppose that $\gamma$ satisfies
\begin{equation}\label{eq4.16}
24(12\sqrt{2}k+4k+1)\gamma<1/2200,
\end{equation}
where $k=K\ell/2$.  
Then
\begin{equation}\label{eq4.17}
d(A',B')\leq  10\gamma+2\beta+13200k\gamma(1+\beta),
\end{equation}
where $\beta=K((1+316800k\gamma)^\ell-1)$.
\end{theorem}

\begin{proof}
Let $M=\overline{A}^w$ and $N=\overline{B}^w$ and write $Z(M)$ and $Z(N)$ for the centres of $M$ and $N$ respectively.  Lemma 5 of \cite{Kadison.Kastler} gives $d(M,N)\leq d(A,B)=\gamma$. By Lemma \ref{Centre}, there is a unitary $u\in (Z(M)\cup Z(N))''$ such that $uZ(M)u^*=Z(uMu^*)=Z(N)$ and
\begin{equation}
\|u_1-I_{\mathcal H}\|\leq 5\gamma.
\end{equation}
Write $M_0=uMu^*$. Then 
\begin{equation}\label{L.2}
d(M_0,N)\leq 2\|u_1-I_{\mathcal H}\|+d(M,N)\leq 11\gamma
\end{equation}
Since $11\gamma<1/10$, Lemma \ref{Decomposition} applies. Thus we can find orthogonal projections $z_{\mathrm{I}_{\text{fin}}},z_{\mathrm{II}_1},z_{\infty}$ in $Z(M_0)$ which sum to 
$ I_{\mathcal H}$ such that:
\begin{itemize}
\item[\rm (i)] $M_0z_{\mathrm{I}_{\text{fin}}}$ and $Nz_{\mathrm{I}_{\text{fin}}}$ are finite type $\mathrm{I}$;
\item[\rm (ii)] $M_0z_\infty$ and $Nz_\infty$ are properly infinite;
\item[\rm (iii)] $M_0z_{\mathrm{II}_1}$ and $Nz_{\mathrm{II}_1}$ are type $\mathrm{II}_1$.
\end{itemize}

Finite type I von Neumann algebras are injective, so have property $D_1$ (\cite[Theorem 2.3]{Christensen.Perturbations2}) while properly infinite algebras have property $D_{3/2}$ (\cite[Theorem 2.4]{Christensen.Perturbations2}). Applying Proposition \ref{LDK.Comm} and Proposition \ref{Prelim.MetricNear} yields
\begin{equation}\label{eqnew2}
 d(M_0'z_{\mathrm{I}_{\text{fin}}},N'z_{\mathrm{I}_{\text{fin}}})\leq 2\cdot2d(M_0z_{\mathrm{I}_{\text{fin}}},Nz_{\mathrm{I}_{\text{fin}}})\leq 4d(M_0,N)\leq 44\gamma,
 \end{equation}
and
 \begin{equation}
d(M_0'z_\infty,N'z_\infty)\leq 2\cdot3d(M_0z_\infty,Nz_\infty)\leq 6d(M_0,N)\leq 66\gamma.
\end{equation}
Choose a maximal family of projections $(z_i)_{i\in \Lambda}$ in $Zz_{\mathrm{II}_1}$ so that each $Zz_i$ has a faithful state. For $\alpha\leq 11\gamma$, the inequality (\ref{Length.II1Case.Est}) follows from (\ref{eq4.16}) so the pairs $M_0z_i$ and $Nz_i$ satisfy the hypothesis of Lemma \ref{Length.II1Case} for each $i$.  The estimates of this lemma then give
\begin{equation}
d(M_0'z_i,N'z_i)\leq 2\beta+13200k\gamma(1+\beta).
\end{equation}
Combining all these cases gives the estimate
\begin{equation}
d(M_0',N')\leq 2\beta+13200k\gamma(1+\beta).
\end{equation}
We then use the estimate $d(M_0',M')\leq 10\gamma$ to obtain
\begin{equation}
d(M',N')\leq 10\gamma+2\beta+13200k\gamma(1+\beta),
\end{equation}
exactly as required.
\end{proof}

In order to use Theorem \ref{Length.Comm} to show that the property of having finite length transfers to close subalgebras, we need one final ingredient detailing how the local distance property behaves for close C$^*$-algebras with close commutants.\begin{lemma}\label{Length.PertLDK}
Let $A$ and $B$ be $C^*$-algebras on a Hilbert space $\mathcal H$. Suppose that $d(A,B)<\gamma$ and $d(A',B')<\eta$. Suppose that $A$ has propertry $LD_k$, where $2\eta+2k\gamma<1$. Then $B$ has property $LD_\const $ where
\begin{equation}\label{eq40.6}
 \const  = \frac{k}{1-2\eta-2k\gamma}.
\end{equation}
\end{lemma}

\begin{proof}
Consider an element $x\in \mathbb B(\mathcal H)\setminus B'$. By scaling we may assume that $\|\ad(x)|_B\| = 1$. By ultraweak compactness, there exists $b'\in B'$ so that $\|x-b'\| = d(x,B')$. The replacement of $x$ by $x-b'$ allows us to make the further assumption that $\|x\| = d(x,B')$. Our objective  now is to estimate $\|x\|$ from above.

Consider $a\in A, \|a\|\le 1$, and choose $b\in B$, $\|b\|\le 1$, so that $\|a-b\|<\gamma$. Then
\begin{equation}\label{eq40.7}
\|[x,a]\| \le \|[x,b]\| + \|[x,a-b]\|
\le 1+2\gamma\|x\|.
\end{equation}
Thus $\|\ad(x)|_A\|\le 1 +2\gamma\|x\|$. Let $T = \{t\in A'\colon \|x-t\| \le \|x\|\}$, non-empty since $0\in T$. The triangle inequality shows that each $t\in T$ satisfies $\|t\|\le 2\|x\|$. For each $t\in T$, choose $s\in B'$ so that $\|t-s\|\leq\eta\|t\|\le 2\eta\|x\|$. Then
\begin{equation}\label{eq40.8}
 \|x-t\| \ge \|x-s\| - \|t-s\| \ge \|x\|-2\eta\|x\|.
\end{equation}
Letting $t\in T$ vary, this yields
\begin{equation}\label{eq40.9}
 d(x,A') \ge (1-2\eta)\|x\|.
\end{equation}
Since $A$ has property $LD_k$, we obtain
\begin{equation}\label{eq40.10}
 (1-2\eta)\|x\| \le d(x,A') \le k\|\ad(x)|_A\| \le k+2k\gamma\|x\|.
\end{equation}
This implies that
\begin{equation}\label{eq40.11}
 \|x\| \le \frac{k}{1-2\eta-2k\gamma}.
\end{equation}
Since we also have $\|x\| = d(x,B')$ and $\|\ad(x)|_B\|=1$, this last inequality states that  $B$ has property $LD_\const $ for
\begin{equation}\label{eq40.12}
 \const  = \frac{k}{1-2\eta-2k\gamma},
\end{equation}
completing the proof.
\end{proof}

We are now in a position to establish the main result of this section: that C$^*$-algebras sufficiently close to those of finite length also have finite length.
\begin{theorem}\label{Length.DH}
Let $C$ be a $C^*$-algebra and let $A$ and $B$ be two $C^*$-subalgebras of $C$. Suppose that $d(A,B)<\gamma$, and  that $A$ has finite length at most $\ell$ with length constant at most $K$. Write $k=K\ell/2$,
\begin{equation}\label{Length.DH.Beta}
\beta=K\Big((1+316800k\gamma+528\gamma)^{\ell}-1\Big),
\end{equation}
and
\begin{equation}\label{Length.DH.Eta}
\eta=10\gamma+2\beta+13200k\gamma(1+\beta).
\end{equation}
If the inequalities
\begin{equation}\label{Length.DH.1}
24(12\sqrt{2}k+4k+1)\gamma<1/2200,\quad 2\eta+2k\gamma<1
\end{equation}
are satisfied, then $B$ has property $D_\const $ for
\begin{equation}\label{Length.DH.2}
\const =\frac{k}{1-2\eta-k\gamma}.
\end{equation}
In particular $B$ has finite length and the length of $B$ is at most $\lfloor 2\const \rfloor\leq
\lfloor K\ell \rfloor$.
\end{theorem}

\begin{proof}
Let $\pi:B\rightarrow\mathbb B(\mathcal K)$ be a representation of $B$ on a Hilbert space $\mathcal K$, and let $\rho:C\rightarrow\mathbb B(\mathcal H)$ be a representation of $C$ on a larger Hilbert space $\mathcal H$ so that $\rho$ extends $\pi$ (see \cite[Proposition II.6.4.11]{Blackadar.OperatorAlgebras}, for example). Since $\rho(A)$ is a quotient of $A$, the length of $\rho(A)$ is at most $\ell$ with length constant at most $K$. Theorem \ref{Length.Comm} gives
\begin{equation}\label{Length.DH.3}
d(\rho(A)',\rho(B)')< 10\gamma+2\beta+13200k\gamma(1+\beta)=\eta,
\end{equation}
where the first inequality of (\ref{Length.DH.1}) is the estimate required to apply Theorem \ref{Length.Comm}. The second inequality of (\ref{Length.DH.1}) is the hypothesis of Lemma \ref{Length.PertLDK} and so $\pi(B)$ has property $LD_\const $, where $\const $ is given by (\ref{Length.DH.2}).
Since the representation $\pi$ of $B$ was arbitrary, $B$ has property $D_\const $, as required.  The final statement of the Theorem follows from Proposition \ref{Prelim.DKLength}.
\end{proof}

\begin{remark} 
While it is obvious that sufficiently small choices of
$\gamma$ will allow us to satisfy \eqref{Length.DH.1}, the dependence of
these inequalities on $K$ and $\ell$ does not make clear
the range of admissible values for this constant. We
consider here one example. Suppose that $\ell =3$ and $K=1$
for $A$,
a situation that occurs when $A$ is a
stable but non-nuclear C$^*$-algebra, for instance. Then direct calculation shows
that \eqref{Length.DH.1} is satisfied for $\gamma < 10^{-7}$.$\hfill\square$
\end{remark}

We now turn to some immediate applications of Theorem \ref{Length.DH}. The first corollary follows from \cite[Theorem 3.1]{Christensen.NearInclusions}.
\begin{corollary}\label{Length.Cor}
Let $A$ and $B$ be C$^*$-subalgebras of some C$^*$-algebra $C$ and let $E$ be a nuclear C$^*$-algebra.  Suppose that $A$ has finite length at most $\ell$ and length constant at most $K$ and suppose that $d(A,B)<\gamma$.  Let $k,\beta,\eta$ be as in Theorem \ref{Length.DH}. If $\gamma$ satisfies the inequalities (\ref{Length.DH.1}), then
\begin{equation}
B\otimes E\subset_\mu A\otimes E,
\end{equation}
where
\begin{equation}
\mu=\frac{6k\gamma}{1-2\eta-k\gamma}.
\end{equation}
In particular, $d(A\otimes E,B\otimes E)< 2\mu$.
\end{corollary}
Using the fact that the distance between any two C$^*$-subalgebras of the same C$^*$-algebra is at most $1$, we get the following alternative formulation of the previous corollary.
\begin{corollary}
For each $\ell\geq 1$ and $K\geq 1$ there exists a constant $L_{\ell,K}$ (which can be found explicitly) such that whenever $A$ and $B$ are C$^*$-subalgebras of some C$^*$-algebra $C$ such that $A$ has length at most $\ell$ and length constant at most $K$, then
\begin{equation}
d(A\otimes E,B\otimes E)\leq L_{\ell,K}d(A,B)
\end{equation} 
for every nuclear C$^*$-algebra $E$.
\end{corollary}

Raeburn and Taylor \cite{Raeburn.CohomologyPerturbation} showed the existence of a constant $\gamma_0>0$ with the property that if two von Neumann algebras $M$ and $N$ have $d(M,N)<\gamma_0$, then $M$ is injective if and only if $N$ is injective.  As a consequence (using \cite[Lemma 5]{Kadison.Kastler} and that a C$^*$-algebra $A$ is nuclear if and only if $A^{**}$ is injective), it follows that two C$^*$-algebras $A$ and $B$ with $d(A,B)<\gamma_0$ are either both nuclear or both non-nuclear.  This argument was given in \cite[Theorem 6.5]{Christensen.NearInclusions}, in which it was also shown that one can take $\gamma_0=1/101$.  Finite dimensional C$^*$-algebras have length $1$ (with length constant $1$). In \cite{Pisier.SimilarityNuclear}, Pisier characterised nuclearity using the similarity length showing that a C$^*$-algebra is nuclear if and only if it has length at most $2$ (it then follows that the length constant must be $1$). We can use this characterisation and Theorem \ref{Length.DH} to recapture the stablity of nuclearity under small perturbations: if we take $\ell=2$ and $K=1$ in Theorem \ref{Length.DH}, then there is certainly a constant $\gamma_0$ for which $\gamma<\gamma_0$ satisfy (\ref{Length.DH.1}) and the $\const $ given in (\ref{Length.DH.2}) has $\const <3/2$.  It follows that if $A$ is nuclear and $d(A,B)<\gamma_0$, then $B$ has length at most $2$ so is nuclear.   We obtain a similar statement for algebras of higher lengths, though as our results do not enable us to control the length constant we must restrict to the case of length constant $1$, (although no example of a C$^*$-algebra with finite length and length constant strictly larger than $1$ is known). In the case $\ell(A)=3$ and $d(A,B)<\gamma_0$, we obtain the exact value $\ell(B)=3$, since any smaller value would imply nuclearity of $B$ and hence of $A$, by the preceding remarks. This would give the contradiction $\ell(A)\leq 2$. We record this discussion in the following result, using the notation above.
\begin{corollary}
For each $\ell\geq 1$, there exists a constant $\gamma_\ell>0$ such that if $A$ and $B$ are two C$^*$-subalgebras of a C$^*$-algebra $C$ with $d(A,B)<\gamma_\ell$ and $A$ has length at most $\ell$ with length constant at most $1$, then $B$ has length at most $\ell$. In particular, if $\ell(A)=3$ and $\gamma_3$ is chosen to be less than $\gamma_0$, then $\ell(B)=3$.
\end{corollary}

If we are not concerned about the exact value of the length and only consider whether C$^*$-algebras have the similarity property, we obtain the next corollary. 
\begin{corollary}
Let $C$ be a C$^*$-algebra.  The set of C$^*$-subalgebras with the similarity property is an open subset in the Kadison-Kastler metric $d(\cdot,\cdot)$.
\end{corollary}

Finally note that one can define a natural pseudometric $d_{\text{cb}}(\cdot,\cdot)$ on the set of all C$^*$-subalgebras of a C$^*$-algebra $A$ by
$$
d_{\text{cb}}(A,B)=\sup_{n\in\mathbb N}d(A\otimes \mathbb M_n,B\otimes \mathbb M_n),
$$
where we measure the distance between $A\otimes\mathbb M_n$ and $B\otimes\mathbb M_n$ in $C\otimes\mathbb M_n$.  This is only a pseudometric as in general there is no reason why $d_{\text{cb}}(A,B)<\infty$.  However Corollary 4.7 shows that for each $\ell\geq 1$ and $K\geq 1$, $d_{\text{cb}}(\cdot,\cdot)$ is equivalent to $d(\cdot,\cdot)$ provided we work on the open set of those C$^*$-algebras sufficiently close to one of length at most $\ell$ and length constant at most $K$.  In the next section, this observation enables us to show that on these sets sufficiently close C$^*$-algebras have isomorphic $K$-theories.

\section{$K$-theory and traces}\label{KTrace}

The classification programme for nuclear C$^*$-algebras was introduced by Elliott in  \cite{Elliott.ClassificationAF}, in which separable AF C$^*$-algebras were classified by their \emph{local semigroups},  the Murray-von Neumann equivalence classes of projections with addition defined where it makes sense.  In \cite{Phillips.PerturbationAF}, J. Phillips and Raeburn showed that sufficiently close C$^*$-algebras have isomorphic local semigroups and deduced that sufficiently close separable AF C$^*$-algebras must be isomorphic. Subsequently, Khoshkam examined the $K$-theory of close subalgebras in \cite{Khoshkam.PerturbationK}, showing that sufficiently close nuclear C$^*$-algebras have isomorphic $K$-groups and so opened the road to using classification results to resolve perturbation problems.  As Khoshkam notes, the argument of \cite{Khoshkam.PerturbationK} only uses nuclearity to lift  a near inclusion $A\subset_\gamma B$ with $A$ nuclear to near inclusions $A\otimes\mathbb M_n\subset_{6\gamma}B\otimes\mathbb M_n$ for all $n$, via property $D_1$ and \cite[Theorem 3.1]{Christensen.NearInclusions}.

Let $A,B$ be C$^*$-subalgebras of a C$^*$-algebra $C$.  Write $\tilde{C}$ for the C$^*$-algebra obtained by adding a new unit $I$ to $C$ (even if $C$ already has a unit) and let $\tilde{A}=C^*(A,I)$ and $\tilde{B}=C^*(B,I)$ so that $\tilde{A}$ and $\tilde{B}$ share the same unit.  Recall that $K_0(A)$ is the kernel of the natural map $K_0(\tilde{A})\rightarrow K_0(\mathbb CI)\cong\mathbb Z$ and that $K_1(A)$ is naturally isomorphic to $K_1(\tilde{A})$ (as $K_1(\mathbb CI)=\{0\}$).  

\begin{theorem}[Khoshkam --- {\cite[Proposition 2.4, Remark 2.5]{Khoshkam.PerturbationK}}]\label{Trace.K}
Let $A,B$ be C$^*$-subalgebras of a C$^*$-algebra $C$. Suppose that there exists $\gamma\leq 1/3$ such that $A\otimes\mathbb M_n\subset_\gamma B\otimes\mathbb M_n$ for all $n\in\mathbb N$.  Then there are homomorphisms $\Phi_0:K_0(A)\rightarrow K_0(B)$ and $\Phi_1:K_1(A)\rightarrow K_1(B)$ defined as follows.
\begin{itemize}
\item[\rm (i)] Given a projection $p\in \mathbb M_n(\tilde{A})$, choose a projection $q\in \mathbb M_n(\tilde{B})$ with $\|p-q\|<\sqrt{2}\gamma$. Define $\Phi_0([p]_0)=[q]_0$.  This is well defined and extends to a homomorphism $K_0(\tilde{A})\rightarrow K_0(\tilde{B})$ which induces a homorphism $\Phi_0:K_0(A)\rightarrow K_0(B)$.
\item[\rm (ii)] Given a unitary $u\in \mathbb M_n(\tilde{A})$, choose a unitary $v\in\mathbb M_n(\tilde{B})$ with $\|u-v\|<\sqrt{2}\gamma$. Define $\Phi_1([u]_1)=[v]_1$.  This is well defined and extends to a homomorphism $K_1(A)\cong K_1(\tilde{A})\rightarrow K_1(\tilde{B})\cong K_1(B)$.
\end{itemize}
If, in addition, $B\otimes\mathbb M_n\subset_{\gamma'} A\otimes\mathbb M_n$ for some $\gamma'\leq 1/3$ and for all $n\in \mathbb N$, then $\Phi_0$ and $\Phi_1$ are isomorphisms.
\end{theorem}
\noindent The choices required to define the maps above can be made. The remarks of \cite[1.4]{Khoshkam.PerturbationK} show that if $A\otimes\mathbb M_n\subset_\gamma B\otimes\mathbb M_n$, then $\tilde{A}\otimes\mathbb M_n\subset_{2\gamma}\tilde{B}\otimes\mathbb M_n$. The estimates given in Proposition \ref{Prelim.Estimates} can then be used to make the necessary choices.   

\begin{remark}
The map $\Phi_0$ also preserves the order structure of $K_0$.  Write $K_0(A)^+$ for the positive cone in $K_0(A)$ which consists of the classes $[p]_0$ in $K_0(A)$ corresponding to projections $p$ in $\mathbb M_n(A)$ for some $n$ and write $\Sigma(A)$ for the \emph{scale} in $K_0(A)$ which consists of the classes $[p]_0$ in $K_0(A)$ corresponding to projections in $A$.  Then, provided the $\gamma$ of the previous theorem satisfies $(2+\sqrt{2})\gamma<1$, it follows that $\Phi_0$ has $\Phi_0(K_0(A)^+)\subseteq K_0(B)^+$ and $\Phi_0(\Sigma(A))\subseteq \Sigma(B)$.  For every projection $p$ in some $\mathbb M_n(A)$, there is a projection $q_0\in\mathbb M_n(B)$ with $\|p-q_0\|\leq 2\gamma$ by  Proposition \ref{Prelim.Estimates} (i).  By definition $\Phi_0([p]_0)=[q]_0$, where $q$ is a projection in $\mathbb M_n(\tilde{B})$ with $\|p-q_0\|<\sqrt{2}\gamma$. The condition on $\gamma$ ensures that $\|q-q_0\|<1$ so $[q]_0=[q_0]_0\in K_0(B)^+$. 
\end{remark}

Combining Khoshkam's work with our analysis in section \ref{Length} gives the following general result, showing that sufficiently close algebras have isomorphic (ordered) $K$-theories provided one algebra has finite length.

\begin{corollary}\label{K.Sim}
Let $A$ and $B$ be C$^*$-subalgebras of a C$^*$-algebra $C$.  Suppose that $A$ has length at most $\ell$ and length constant at most $K$. Suppose further that $d(A,B)=\gamma$ for some $\gamma$ satisfying (\ref{Length.DH.1}) and such that the $\mu$ of Corollary \ref{Length.Cor} satisfies $\mu<1/(2+\sqrt{2})$.  Then $\Phi_*:K_*(A)\rightarrow K_*(B)$ is an isomorphism preserving the order structure and scale on $K_0$.
\end{corollary}

In the finite case $K$-theory alone is not sufficient to classify a large class of simple separable nuclear C$^*$-algebras and so the Elliott invariant has been expanded to include tracial information. In the (finite) non-unital case, the Elliott invariant consists of the data
\begin{equation}
((K_0(A),K_0(A)^+,\Sigma(A)),K_1(A),T(A),\rho_A),
\end{equation}
where $T(A)$ is the cone of positive tracial functionals on $A$ and $\rho_A$ the natural pairing $K_0(A)\times T(A)\rightarrow \mathbb R$ given by extending $([p]_0,\tau)\mapsto (\tau\otimes \mathrm{tr}_n)(p)$, when $p$ is a projection in $A\otimes\mathbb M_n$ and $\mathrm{tr}_n$ is the unique trace on $\mathbb M_n$ with $\mathrm{tr}_n(I_{\mathbb M_n})=n$.  We refer to \cite{Rordam.ClassificationBook} for a discussion of these invariants and an account of the classification programme.

In the rest of this section our objective is to examine traces on close C$^*$-algebras.   Suppose we are given a near inclusion $A\subset_\gamma B$ of unital C$^*$-algebras which share the same unit and a tracial state $\tau$ on $B$.  This induces a state $K_0(\tau)$ on $K_0(B)$. If $A$ has finite length and $\gamma$ is sufficiently small, then we can obtain a state $K_0(\tau)\circ\Phi_0$ on $K_0(A)$ by composing with Khoshkam's map $\Phi_0:K_0(A)\rightarrow K_0(B)$ of Theorem \ref{Trace.K}.  By Theorem 3.3 of \cite{Blackadar.ExtendingStates}, $K_0(\tau)\circ\Phi_0$ arises from a quasitrace $A$.  If $A$ is additionally assumed to be exact, then Haagerup's result \cite{Haagerup.Quasitrace} shows that this quasitrace is actually a trace on $A$.  In this way we obtain a map from the tracial states on $B$ into those on $A$. We prefer a more direct approach passing through the bidual, which gives an isomorphism between the trace states of close C$^*$-algebras without assuming exactness.  

\begin{lemma}\label{KTrace.Lemma}
Suppose that $A$ and $B$ are C$^*$-subalgebras of some C$^*$-algebra $C$ such that $d(A,B)=\gamma$ for some $\gamma<1/2200$.  Then there exists an affine isomorphism $\Psi:T(B)\rightarrow T(A)$. Furthermore, given $n\in\mathbb N$ and projections $p\in A\otimes\mathbb M_n$ and $q\in B\otimes\mathbb M_n$ with $\|p-q\|<1/2-10\gamma$, then 
\begin{equation}\label{Trace.1}
(\Psi(\tau)\otimes\mathrm{tr}_n)(p)=(\tau\otimes\mathrm{tr}_n)(q),\quad \tau\in T(B).
\end{equation}
\end{lemma}

\begin{proof}
Working in the universal representation of $C$, the weak closures $M$ and $N$ of $A$ and $B$ are isometrically isomorphic to $A^{**}$ and $B^{**}$ respectively. Lemma 5 of \cite{Kadison.Kastler} gives $d(M,N)\leq d(A,B)$.  By Lemma \ref{Centre}, there exists a unitary $u\in (Z(M)\cup Z(N))''$ such that $Z(uMu^*)=Z(N)$ and $\|u-I\|\leq 5\gamma$. Write $A_1=uAu^*$ and $M_1=uMu^*$ so that $d(M_1,N)\leq 11\gamma$.  Since $11\gamma<1/10$, Lemma \ref{Decomposition} applies.  In particular, there is a projection $z_{\text{fin}}$ in $Z(M_1)=Z(N)$ such that $Mz_{\text{fin}}$ and $Nz_{\text{fin}}$ are both finite while $M(I-z_{\text{fin}})$ and $N(I-z_{\text{fin}})$ are both purely infinite.

Now take a positive linear tracial functional $\tau$ on $B$. There is a unique extension $\tau^{**}$ to a normal positive linear tracial functional on $N$. This must factor through the finite part of $N$ and the centre valued trace on this algebra (see \cite[8.2]{KR.2}).  It follows that there is a unique positive normal functional $\phi_\tau$ on $Z(N)z_{\text{fin}}$ such that
\begin{equation}
\tau^{**}(x)=\phi_\tau(\ct_{Nz_{\text{fin}}}(x)),\quad x\in N,
\end{equation}
where $\ct_{Nz_{\text{fin}}}$ is the centre valued trace on $Nz_{\text{fin}}$. Then $\phi_\tau\circ\ct_{M_1z_{\text{fin}}}$ defines a normal positive tracial functional on $M_1$. Define a positive linear functional $\tau_1$ on $A_1$ by restricting this functional to $A_1$, i.e.
\begin{equation}
\tau_1=(\phi_\tau\circ\ct_{M_1z_{\text{fin}}})|_{A_1}.
\end{equation}
Let $\Psi(\tau):A\rightarrow\mathbb C$ be given by $\Psi(\tau)(x)=\tau_1(uxu^*)$ so $\Psi(\tau)$ is a positive tracial functional on $A$.  
The map $\Psi$ is evidently affine. Since every positive tracial functional on $A_1$ extends uniquely to $M_1$, where it factors through $M_1z_{\text{fin}}$ and $\ct_{M_1z_{\text{fin}}}$ the map $\Psi$ is onto and so an affine isomorphism between the positive tracial functionals on $B$ and those on $A$.

We now establish (\ref{Trace.1}). Fix $n\in\mathbb N$ and projections $p\in A\otimes\mathbb M_n$ and $q\in B\otimes\mathbb M_n$ with $\|p-q\|<1/2-10\gamma$.  Note that $\tau\otimes\tr_n$ is the restriction of $(\phi_\tau\otimes\tr_n)\circ\ct_{Nz_{\text{fin}}\otimes\mathbb M_n}$ to $B\otimes\mathbb M_n$, while $\tau_1$ gives rise to $\tau_1\otimes\tr_n$ on $A_1\otimes\mathbb M_n$ which is given by restricting $((\phi_\tau\otimes\tr_n)\circ\ct_{M_1z_{\text{fin}}\otimes\mathbb M_n})$ to $A_1\otimes\mathbb M_n$. Then 
\begin{equation}
\|(u\otimes I_{\mathbb M_n})p(u\otimes I_{\mathbb M_n})^*-q\|\leq 2\|u-I\|+\|p-q\|<1/2.
\end{equation}
As $d(M_1,N)\leq 11\gamma<1/200$, Lemma \ref{CentralTrace} applies and so
\begin{equation}
\ct_{M_1z_{\text{fin}}\otimes\mathbb M_n}(upu^*)=\ct_{Nz_{\text{fin}}\otimes\mathbb M_n}(q).
\end{equation}
Thus $(\tau\otimes\tr_n)(q)=(\tau_1\otimes\tr_n)(upu^*)=(\Psi(\tau)\otimes\tr_n)(p)$.
\end{proof}

Combining the previous lemma with the results of Section \ref{Length}, it follows that sufficiently close C$^*$-algebras have the same Elliott invariant when one has finite length.  
\begin{theorem}\label{KTrace.Main}
Suppose that $A$ and $B$ are C$^*$-subalgebras of some C$^*$-algebra $C$. Write $d(A,B)=\gamma$ and suppose that $A$ has length at most $\ell$ with length constant at most $K$.  Suppose $\gamma$ satisfies the inequalities (\ref{Length.DH.1}) and the $\mu$ of Corollary \ref{Length.Cor} satisfies $\mu<1/(2+\sqrt{2})$. Then there exist isomorphisms $\Phi_*:K_*(A)\rightarrow K_*(B)$ between the ordered $K$-theories of $A$ and $B$, which preserve the scale and an affine isomorphism $\Psi:T(B)\rightarrow T(A)$ such that \begin{equation}
\rho_A(x,\Psi(\tau))=\rho_B(\Phi_0(x),\tau)),\quad x\in K_0(A),\ \tau\in T(B).
\end{equation}
\end{theorem}

\section{Kirchberg algebras and real rank zero}\label{Kirchberg}

A Kirchberg C$^*$-algebra is defined by the properties of being nuclear, purely infinite, simple, and separable. One of the crowning achievements of Elliott's classification programme is the theorem of Kirchberg and C. Phillips which shows that Kirchberg algebras satisfying the UCT are classifiable by their $K$-theory \cite{Kirchberg.Phillips,Kirchberg.ClassificationBook}. In this section we make use of this result to examine perturbation theory for such C$^*$-algebras.  Our  objective is to show that any C$^*$-algebra sufficiently close to a Kirchberg algebra is again a Kirchberg algebra.  By earlier results of Christensen and J. Phillips this amounts to showing that a C$^*$-algebra sufficiently close to a simple separable purely infinite algebra is itself purely infinite. This result can be established directly, but we prefer to use a characterisation due to Zhang \cite{Zhang.PurelyInfinite}. He shows that a simple C$^*$-algebra is purely infinite if and only if it is real rank zero and every non-zero projection is infinite.  We will show that these two properties transfer to sufficiently close algebras. This has the advantage of additionally establishing a perturbation result for the property of being real rank zero, which is also of importance in the classification programme for finite C$^*$-algebras. We begin with the second of the two properties in Zhang's characterisation above.

\begin{lemma}\label{Kirchberg.Infinite}
Let $A$ and $B$ be C$^*$-subalgebras of a C$^*$-algebra $C$ with $d(A,B)<1/14$.  If every non-zero projection in $A$ is infinite, then every non-zero projection in $B$ is infinite.
\end{lemma}
\begin{proof}
Take $\gamma>0$ with $d(A,B)<\gamma<1/14$. Given a non-zero projection $p$ in $B$, use   Proposition \ref{Prelim.Estimates} (i) to find a projection $q\in A$ with $\|p-q\|<2\gamma$ so that $q$ is non-zero.  By hypothesis $q$ is infinite so there exists a partial isometry $v\in A$ with $vv^*<q$ and $v^*v=q$.  Take an operator $b_0$ in the unit ball of $B$ with $\|b_0-v\|<\gamma$ and define $b=pb_0p$. Then 
\begin{align}
\|v-b\|\leq&\|qvq-pvq\|+\|pvq-pb_0q\|+\|pb_0q-pb_0p\|\notag\\
\leq&\|p-q\|+\|v-b_0\|+\|p-q\|\leq 5\gamma.
\end{align}
Now represent $C$ on a Hilbert space $\mathcal H$. Let $y=b+(I-p)$ and $x=v+(I-q)$ so that $\|y-x\|\leq \|v-b\|+\|p-q\|<7\gamma$.  Since $x$ is an isometry, we have 
\begin{equation}\label{eqgamma}
(1-7\gamma)\|\xi\|\leq \|y\xi\|=\||y|\xi\|,\quad \xi\in \mathcal H
\end{equation}
and so $|y|\geq (1-7\gamma)I$. As $\gamma<1/14$, the operator $|y|$ is invertible in $C^*(B,I)$. Thus, in the polar decomposition $y=w_0|y|$, the partial isometry $w_0$ lies in $C^*(B,I)$.  If $y$ were invertible then we would have $\|y^{-1}\|\leq (1-7\gamma)^{-1}$ from \eqref{eqgamma}. Then
\begin{equation}
\|I_{\mathcal H}-y^{-1}x\|=\|y^{-1}(y-x)\|<7\gamma/(1-7\gamma)<1,
\end{equation}
since $\gamma< 1/14$. Thus $y^{-1}x$ is invertible so $x$ is invertible. This contradiction shows that $y$ is not invertible, and consequently $w_0$ is not a unitary. Now let $w=p w_0 p$. The definition of $y$ implies that
$p$ commutes with $y$, and hence with $|y|$ and
$w_0=y|y|^{-1}$. Since $w_0^*w_0=I_{\mathcal H}$, we see that
$w^*w={w_0}^*w_0p=p$, so that $w$ is a partial isometry. On
the other hand, $ww^*=pw_0w_0^*p\leq p$. If equality held, then $w_0$ would be a unitary, since $(I_{\mathcal H}-p)w_0w_0^*=I_{\mathcal H}-p$ due
to the invertibility of $(I_{\mathcal H}-p)y$ on $(I_{\mathcal H}-p)(\mathcal H)$.
This contradiction shows that $ww^*< p$, and proves that $p$
is
an infinite projection.
\end{proof}

Recall that a C$^*$-algebra $A$ has \emph{real rank zero} if the self-adjoint operators in $A$ of finite spectrum are dense in the self-adjoint operators of $A$.  Equivalently, a C$^*$-algebra $A$ has real rank zero if and only if it has the \emph{hereditary property} that every hereditary C$^*$-subalgebra of $A$ has an approximate unit of projections, see \cite[V.3.29]{Blackadar.OperatorAlgebras}. This latter condition is amenable to perturbation arguments.
\begin{lemma}\label{Kirchberg.Hereditary}
Suppose that $A$ and $B$ are C$^*$-subalgebras of a unital C$^*$-algebra $C$, and let $\gamma$ satisfy  $d(A,B)<\gamma<1/8$. If $A$ has real rank zero, then for all $k\geq 0$ in $B$, there exists a projection $p\in \overline{kB k}$ such that
\begin{equation}
\|k-kp\|\leq 7\gamma\|k\|\leq (7/8)\|k\|.
\end{equation}
\end{lemma}
\begin{proof}
Suppose that $C$ is faithfully represented on $\mathcal H$. 
Let $k\geq 0$ lie in $B$, and assume without loss of
generality that $\|k\|=1$. Then choose $h\in A_{{\text{s.a.}}}$ such that
$\|h\|\leq 1$, $\|h-k\|<\gamma$, and the spectrum of $h$ is
finite (and is contained in $[-\gamma,1]$). Let $q\in A$ be the spectral projection of $h$ for
the interval $[1/2,1]$ and choose, by Proposition \ref{Prelim.Estimates} (i), a projection
$r\in B$ with $\|r-q\|<2\gamma$. Then
\begin{align}
\|k(I-r)\|&\leq \|(k-h)(I-r)\|+\|h(q-r)\|+\|h(I-q)\|\notag\\
&<\gamma+2\gamma+1/2=3\gamma+1/2<7/8.\label{eq6.500}
\end{align}
For a unit vector $\xi\in r\mathcal H$, we have the
inequality
\begin{align}
\|k\xi\|&\geq \|hq\xi\|-\|h(r-q)\xi\|-\|(h-k)\xi\|\notag\\
&\geq \|q\xi\|/2-2\gamma-\gamma\notag\\
&\geq \|r\xi\|/2-4\gamma=1/2-4\gamma>0.
\end{align}
Thus $kr$ is bounded below on $r\mathcal H$, so the operator
$t=|kr|=(rk^2r)^{1/2}$ is invertible on $r\mathcal H$. Let
$kr=vt$ be the polar decomposition, where $v\in \mathbb
B(\mathcal H)$ is a partial isometry. Using the
invertibility of $t$ on $r\mathcal H$, it is easy to check
that $v$ is the norm limit of the sequence
$\{v_n\}_{n=1}^{\infty}$ whose elements are defined by
$v_n=kr(t+n^{-1}I)^{-1}$ for $n\geq 1$. This shows that $v\in
B$, so the range projection $p=vv^*$ of $kr$ also lies in
$B$. Moreover, $p\in \overline{kBk}$ since this algebra
contains each element $v_nv_n^*$, $n\geq 1$. By
construction, $(I-p)kr=0$, so
\begin{equation}
\|k(I-p)\|=\|(I-p)k\|=\|(I-p)k(I-r)\|<7\gamma<7/8,
\end{equation}
from \eqref{eq6.500}.
\end{proof}

\begin{theorem}\label{Kirchberg.RealRankZero}
Let $A$ and $B$ be C$^*$-subalgebras of a C$^*$-algebra $C$ with $d(A,B)< 1/8$. Then $A$ has real rank zero if and only if $B$ has real rank zero.
\end{theorem}
\begin{proof}
Fix a hereditary $C^*$-subalgebra $E$ of $B$.  Given $x_1,\dots,x_n\in E$ and $\varepsilon>0$, we must find a projection $p\in E$ with $\|x_i-x_ip\|<\varepsilon$ for all $i$. As in the proof of \cite[Theorem V.7.3]{Davidson.Example}, by taking $x=\sum_ix_i^*x_i$, we have
\begin{equation}
\|x_i-x_ip\|^2=\|(I-p)x_i^*x_i(I-p)\|\leq\|(I-p)x(I-p)\|.
\end{equation}
Therefore it suffices to consider a single positive element $x\in E$ and find a
projection $p\in E$ with $\|(I-p)x(I-p)\|<\varepsilon$.

Let $E_0=\overline{xB x}$, the hereditary subalgebra of $B$ generated
by $x$ so that $E_0\subseteq E$. Use Lemma \ref{Kirchberg.Hereditary} to find a
projection $p_1\in E_0$ with $\|x-xp_1\|\leq (7/8) \|x\|$ and so
\begin{equation}
\|(I-p_1)x(I-p_1)\|\leq\|x-xp_1\|\leq (7/8) \|x\|.
\end{equation}
The element $(I-p_1)x(I-p_1)$ is a positive element of $E_0$ and so
generates a hereditary subalgebra
$E_1=\overline{(I-p_1)x(I-p_1)B(I-p_1)x(I-p_1)}$ of $E_0$. We can
then use Lemma \ref{Kirchberg.Hereditary} again to find a projection $p_2\in E_1$ with
\begin{equation}
\|(I-p_1)x(I-p_1)(I-p_2)\|<(7/8)\,\|(I-p_1)x(I-p_1)\|< 
 (7/8)^2 \|x\|.
\end{equation}
Since $E_1\subset \overline{(I-p_1)B (I-p_1)}$, it follows that
$p_2\leq 1-p_1$. Thus $(I-p_1)(I-p_2)=I-(p_1+p_2)$ and we have
\begin{equation}
\|(I-(p_1+p_2))x(I-(p_1+p_2))\|< (7/8)^2\|x\|.
\end{equation}
If we continue in this fashion, we will eventually find orthogonal
projections $p_1,\dots,p_n\in E$ such that
\begin{equation}
\|(I-(p_1+\dots+p_n))x(I-(p_1+\dots+p_n))\|< (7/8)^n <\varepsilon,
\end{equation}
exactly as required.  
\end{proof}

We are now in a position to prove the main result of this section using the previous results, work of J. Phillips and work of the first named author.
\begin{theorem}
Let $A$ and $B$ be C$^*$-subalgebras of a C$^*$-algebra $C$ with $d(A,B)<1/101$.  If $A$ is a Kirchberg algebra, then  $B$ is also a Kirchberg algebra.
\end{theorem}
\begin{proof}
Since $d(A,B)<1/80$ and $A$ is simple, Lemma 1.2 of \cite{Phillips.Perturbation} shows that $B$ is simple.  Since $d(A,B)<1/101$ and $A$ is nuclear, Theorem 6.5 of \cite{Christensen.NearInclusions} shows that $B$ is also nuclear.  Since $d(A,B)<1/2$, $B$ is separable (this is folklore, see the comments in the proof of \cite[Theorem 6.1]{Christensen.NearInclusions} for example, or see \cite{Saw.PerturbNuclear} for a proof). Zhang's characterisation of purely infinite C$^*$-algebras shows that $A$ is real rank zero and every non-zero projection of $A$ is infinite so Theorem \ref{Kirchberg.RealRankZero} and Lemma \ref{Kirchberg.Infinite} show that $B$ has the same properties so is purely infinite.
\end{proof}

The following corollary is immediate from the Kirchberg-Phillips classification theorem 
\cite{Kirchberg.Phillips} and Khoshkam's result \cite{Khoshkam.PerturbationK} that sufficiently close nuclear C$^*$-algebras have isomorphic $K$-theory.
\begin{corollary}
Let $A$ and $B$ be C$^*$-algebras of a C$^*$-algebra $C$ with $d(A,B)<1/101$ and suppose that $A$ and $B$ satisfy the UCT.  If $A$ is a Kirchberg algebra, then $A$ is isomorphic to $B$.
\end{corollary}

\section{Questions}\label{Questions}

The most important question in the perturbation theory of operator algebras is undoubtably Kadison and Kastler's original conjecture \cite{Kadison.Kastler} specialised to the cases of von Neumann algebras or separable C$^*$-algebras (thus excluding the examples from \cite{Christensen.CounterExamples}).  It would be very interesting to find any class $\mathcal A$ of non-injective von Neumann algebras or separable but non-nuclear C$^*$-algebras for which algebras sufficiently close to an algebra $A$ in $\mathcal A$ are isomorphic to $A$. We end the paper with three other natural questions which arose during our investigations.
\begin{question}
Does there exists a constant $\gamma_0>0$ such that if $A$ and $B$ are C$^*$-subalgebras of some C$^*$-algebra $C$ with $d(A,B)<\gamma_0$, then $A$ is exact if and only if $B$ is exact?
\end{question}

\begin{question}
Suppose that $\ell\geq 1$ and $K\geq 1$ are given.  Does there exist a constant $\gamma_{\ell,K}>0$ such that if $A$ and $B$ are C$^*$-subalgebras of some C$^*$-algebra $C$ with $d(A,B)<\gamma_{\ell,K}$ and $A$ has length at most $\ell$ with length constant at most $K$, then there is a natural isomorphism $\text{Ext}(A)\rightarrow \text{Ext}(B)$?
\end{question}
\noindent More generally one could also ask how $KK$-theory behaves in the context of close C$^*$-algebras with finite length.

\begin{question}
Are higher values of the real rank stable under small perturabtions?  What happens to the stable rank under small perturbations?
\end{question}

\end{document}